\documentclass[11pt]{amsart}
\usepackage[mathscr]{eucal}
\usepackage[pdftex, bookmarksnumbered, bookmarksopen, colorlinks, citecolor=blue, linkcolor=blue]{hyperref}
\usepackage{amsfonts}
\usepackage{amssymb}
\usepackage{amsmath}
\usepackage{color}
\usepackage{dsfont}
\usepackage{amsthm}
\usepackage{ytableau}
\usepackage{enumerate}
\usepackage{mathtools}
\allowdisplaybreaks 
\newtheorem{theorem}{Theorem}[section]
\newtheorem{lemma}[theorem]{Lemma}

\newtheorem*{conj}{Conjecture}
\theoremstyle{definition}
\newtheorem{definition}[theorem]{Definition}

\newtheorem{proposition}[theorem]{Proposition}

\theoremstyle{remark}
\newtheorem{remark}[theorem]{Remark}

\numberwithin{equation}{section}

\newcommand{\abs}[1]{\lvert#1\rvert}

\def\intslash{\rlap{\kern  .32em $\mspace {.5mu}\backslash$ }\int}
\def\qsl{{\rlap{\kern  .32em $\mspace {.5mu}\backslash$ }\int_{Q_x}}}

\def\off{{\text{\rm off}}}

\def\T{\mathbb T}
\def\B{\mathrm B}

\def\tens{\overline{\otimes}}

\def\R{{\mathbb R}}
\def\Z{{\mathbb Z}}
\def\M{{\mathcal M}}
\def\C{{\mathcal C}}
\def\F{{\mathcal F}}
\def\A{{\mathcal A}}
\def\Q{{\mathcal Q}}

\def\wh{\widehat}

\def\S{\mathbf S}

\def\Norm#1{{ \left\|  #1 \right\| }}

\def\pari{\partial}

\def\dist{{\text{\it dist}}}

\def\supp{{\text{\rm supp}}}

\def\inn#1#2{\langle#1,#2\rangle}

\def\rta{\rightarrow}

\def\card{\text{\rm card}}
\def\lc{\lesssim}

\def\alp{\alpha}

\def\del{\delta}             
\def\eps{\varepsilon}

\def\zet{\zeta}
\def\tet{\theta}

\def\lam{\lambda}

\def\vphi{\varphi}

\def\fr{\frac}

\newcommand{\Be}{\begin{equation}}
\newcommand{\Ee}{\end{equation}}
\newcommand{\Bes}{\begin{equation*}}
\newcommand{\Ees}{\end{equation*}}
\newcommand{\Bsp}{\begin{split}}
\newcommand{\Esp}{\end{split}}
\newcommand{\Bm}{\begin{multline}}
\newcommand{\Em}{\end{multline}}
\newcommand{\Bea}{\begin{eqnarray}}
\newcommand{\Eea}{\end{eqnarray}}
\newcommand{\Beas}{\begin{eqnarray*}}
\newcommand{\Eeas}{\end{eqnarray*}}
\newcommand{\Benu}{\begin{enumerate}}
\newcommand{\Eenu}{\end{enumerate}}
\newcommand{\Bi}{\begin{itemize}}
\newcommand{\Ei}{\end{itemize}}

\def\p_Q{\chi_Q}                 

\def\R{\mathbb{R}}
\def\Z{\mathbb{Z}}
\def\M{\mathcal{M}}
\def\A{\mathcal{A}}
\def\Q{\mathcal{Q}}
\def\vp{\varphi}
\def\lam{\lambda}
\def\lc{\lesssim}
\def\supp{\operatorname{supp}}
\def\E{\mathsf{E}}

\def\eps{\varepsilon}

\begin{document}

\title[Noncommutative Bochner--Riesz means]{Weak type $(1,1)$ boundedness of Bochner--Riesz means at the critical index on quantum tori}
\author[]{Xudong Lai}
\address{Xudong Lai:
Institute for Advanced Study in Mathematics\\
Harbin Institute of Technology\\
Harbin
150001\\
China;
Zhengzhou Advanced Research Institute\\
Harbin Institute of Technology\\
Zhengzhou
450000\\
China}

\email{xudonglai@hit.edu.cn\ xudonglai@mail.bnu.edu.cn}





\begin{abstract}
We establish a weak type $(1,1)$ estimate for the Bochner--Riesz mean operator at the critical index $\del=\fr{d-1}{2}$
on the quantum torus $\T^d_\tet$, where the associated kernel is rough.
\end{abstract}

	\maketitle

\section{Introduction and main result}

In recent years, noncommutative harmonic analysis has developed rapidly,  driven by progress in  harmonic analysis, free probability, noncommutative geometry and geometric group theory (see e.g. \cite{Cad18,CXY13,GJP19,GJP17, JMP14,JMP18,
Mei07,MP09,Par09,XXX16,XXY18,Del26,CGPT23,PdT25}). In this paper, we focus on the Bochner--Riesz means on quantum tori.

Recall that the Bochner--Riesz means on usual tori $\T^d$ are defined by
\Be\label{e:31BRtorus}
\B_R^\del(f)(x)=\sum_{n\in\Z^d}(1-\fr{|n|^2}{R^2})^\del_+\wh{f}(n)e^{2\pi i\inn{n}{x}}
\Ee
where $\del\geq0$, $R>0$, $(x)_+=\max\{x,0\}$ and $\wh{f}(n)=\int_{\T^d}f(x)e^{-2\pi i\inn{n}{x}}dx$.
The central topic of Bochner--Riesz means in harmonic analysis is  the so-called Bochner--Riesz conjecture as follows (see e.g. \cite{Ste93}).
\begin{conj}
Denote  $\B^\del=\B^\del_1$. Suppose $\del>\del_p=\max\{d|\fr{1}{p}-\fr12|-\fr12,0\}$ and $1\leq p\leq\fr{2d}{d+1}$ (or equivalently consider the dual range $\fr{2d}{d-1}\leq p\leq\infty$).  Then we have
\Bes
\|\B^\del(f)\|_{L_p(\T^d)}\lc\|f\|_{L_p(\T^d)}.
\Ees
\end{conj}
This conjecture, together with the local smoothing conjecture, Fourier restriction conjecture  and Kakeya conjecture, is among the top four conjectures in harmonic analysis (see e.g. \cite{Tao99}).
The Bochner--Riesz conjecture in two dimensions was settled by L.~Carleson and P.~Sj\"olin \cite{CS72}, and independently by C.~Fefferman \cite{Fef73}, L.~H\"ormander \cite{Hor73}, and A.~C\'ordoba \cite{Cor79}.
In dimensions \(d\ge 3\), the Bochner--Riesz conjecture remains open for \(p\) close to the endpoint \(\frac{2d}{d+1}\). For some recent progress, we refer the reader to \cite{BG11,GHI19,GOWWZ25} and the references therein. On the critical index $\del=\del_p$, it is conjectured that $\B^{\del}$ is weak type $(p,p)$ for $1\leq p\leq\fr{2d}{d+1}$ whose index is the same as that of Bochner--Riesz conjecture. M. Christ \cite{Chr88,Chr87} first showed that such results hold for $1\leq p<\fr{2d+1}{d+3}$. Later T. Tao \cite{Tao96} proved the case $p=\fr{2d+1}{d+3}$. The full range at the critical index case is only known in two dimensions and was proved later by A. Seeger \cite{See96B}.

On the other hand, motivated by developments in noncommutative harmonic analysis (see, e.g., \cite{Cad18,Mei07,Md17,Par09,Xu07}), Z.~Chen, Q.~Xu and Z.~Yin first introduced the noncommutative Bochner--Riesz means on quantum tori in \cite{CXY13}. Quantum tori are also known as noncommutative tori or rotation algebras (see \cite{Rie90}). One can regard them as noncommutative analogues of classical tori. Quantum tori are basic examples in operator algebras (see \cite{Dav96}) and are interesting objects in noncommutative geometry that have been extensively studied (see, e.g., \cite{Con94,Rie90,Var06}).
The research of  the first systematic work of harmonic analysis on quantum tori was also given in \cite{CXY13}. Following this research line, the author \cite{Lai22} established the full $L_p$ boundedness of Bochner--Riesz means on two dimensional quantum tori. Some related  Fourier restriction and local smoothing estimates on quantum Euclidean space have also been investigated in \cite{HLW23, HLW25}.
Recently M. de la Salle \cite{Del26} discovered some deep connections between the completely bounded approximation property of the noncommutative $L_p$ space of $\mathrm{SL}_{n}(\Z)$ and the Kakeya conjecture or the Bochner--Riesz conjecture. In particular, he proposed the following open problem: does the noncommutative $L_p$ space of $\mathcal{L}\mathrm{SL}_{2d-1}(\Z)$ have the OAP for some $p\neq 2$ and $\fr{2d}{d+1}\leq p\leq\fr{2d}{d-1}$.
Notice that the index here coincides with that in the Bochner--Riesz means conjecture. In fact, the Bochner--Riesz conjecture can be restated as $\B^\del$ is $L_p$ bounded if $\del>0$ and $\fr{2d}{d+1}\leq p\leq\fr{2d}{d-1}$ (see e.g. \cite{Ste93}).

In this paper, we continue to study the Bochner--Riesz means on quantum tori and focus on the weak type $(1,1)$ estimate at the critical index $\del_1=\fr{d-1}{2}$.
Due to the lack of commutativity, the weak type $(1,1)$ estimate theory in noncommutative harmonic analysis is more challenging than the classical one.
This can be seen from the construction in the noncommutative Calder\'on-Zygmund decomposition: One needs to decompose an $L_1$ function $f$ by multiplying on both the left and right side by a Cuculescu decomposition $\sum_k p_k$ since $f$ and $p_k$ do not commute. Then naturally the off-diagonal terms $p_kfp_j$ for $k\neq j$ are new and we require some novel ideas to deal with them.
This may explain why the theory of noncommutative weak type $(1,1)$ estimates seems to be very interesting.
The noncommutative weak type $(1,1)$ estimates for Calder\'on-Zygmund operators with some smooth kernels have been studied in \cite{Par09,MP09,Cad18,HLX23}. However, the Bochner--Riesz means at the critical index can essentially be regarded as singular oscillatory integral operators whose kernels are non-smooth and lie outside the scope of these previous studies in this direction.
In this paper, we develop some new ideas to establish the weak type $(1,1)$ boundedness of the Bochner--Riesz means at the critical index on quantum tori.

To illustrate our main results, we first recall the definition of quantum torus.
Suppose that \(d\geq 2\) and that \(\theta=(\theta_{k,j})_{1\leq k,j\leq d}\) is a real skew-symmetric \(d\times d\) matrix.
The \(d\)-dimensional noncommutative torus \(\A_\theta\) is the universal \(C^*\)-algebra generated by \(d\) unitary operators \(U_1,\dots,U_d\) satisfying the commutation relations
\[
U_kU_j=e^{2\pi i\theta_{k,j}}U_jU_k,\quad 1\leq k,j\leq d.
\]
It is well known that \(\A_\theta\) admits a faithful tracial state \(\tau\).
Let \(\T_\theta^d\) be the weak operator closure of \(\A_\theta\) in the GNS representation associated with \(\tau\).
We call \(\T_\theta^d\) the \(d\)-dimensional quantum torus.
The state \(\tau\) extends to a normal faithful tracial state on \(\T_\theta^d\), still denoted by \(\tau\).
Thus the quantum torus \(\T_\theta^d\) is a deformation of the classical torus \(\T^d\).
For \(1\leq p<\infty\), let \(L_p(\T_\theta^d)\) be the noncommutative \(L_p\)-space associated with the pair \((\T_\theta^d,\tau)\), equipped with the norm
$
\|x\|_{L_p(\T_\theta^d)}=(\tau(|x|^p))^{1/p}.
$
In the following, we consider the Bochner--Riesz mean on the quantum tori which is defined by
\Be\label{e:31BRqtorusi}
\B^\del_R(f)=\sum_{m\in\Z^d}\big(1-|m/R|^2\big)^\del_+\hat{f}(m)U^m,\quad f\in L_p(\T^d_\tet),
\Ee
where $U=(U_1,\cdots,U_d)$, $U^m=U_1^{m_1}\cdots U_d^{m_d}$ and $\wh{f}(m)=\tau((U^m)^*f)$. Let $\B^\del=\B^\del_1$. Our main result is stated as follows.

\begin{theorem}\label{t:31main}
Let $\delta$ be the critical index $\del_1= \frac{d-1}{2}$. Then the Bochner--Riesz operator $\B^\del$ is of weak type $(1,1)$ on $L_1(\T_\tet^d)$. That is, for all $f\in L_1(\T^d_\tet)$ and $\lam>0$, the following estimate holds
\[
\lam\,\tau\bigl(\chi_{(\lam,\infty)}(|\B^\delta f|)\bigr) \lc \|f\|_{L_1(\T_\tet^d)}.
\]
\end{theorem}

\begin{theorem}\label{c:31cor}
Let $\del= \frac{d-1}{2}$. For $f\in L_1(\T^d_\tet)$, $\B^\del_R(f)$ converges to $f$ as $R\rta +\infty$ in measure in the following sense:
\[
\lim_{R\rta+\infty}\tau\bigl(\chi_{(\lam,\infty)}(|\B^\delta_R(f)-f|)\bigr) =0
\]
holds for any $\lambda>0$.
\end{theorem}

The classical (commutative) Bochner--Riesz means at the critical index are known to be of weak type $(1,1)$ by a theorem of Christ \cite{Chr88}. Our main theorem above extends it to the noncommutative setting on quantum tori. The Bochner--Riesz operator $\B^\del$ defined on quantum tori is fully noncommutative and its analysis is difficult.
Our strategy is to transfer its study to the operator-valued setting. This is done by two steps: first transferring it to the operator-valued setting on the usual tori  and then the usual Euclidean space (see Section \ref{s:31trans}).
In the latter case we can do analysis based on some  noncommutative theory of harmonic analysis. After applying these two transference  principles, our goal is  to show that the operator-valued $\B^\del$ is of weak type $(1,1)$ on $L_1(L_\infty(\R^d)\bar\otimes\M)$.

Since the kernel of \(\B^\delta\) is essentially the product of an oscillatory factor and a Calder\'on--Zygmund kernel, it lacks smoothness and in fact is a rough kernel.
It is quite difficult to establish the weak type $(1,1)$ boundedness for noncommutative rough operators.
Up to now, to deal with the noncommutative weak type $(1,1)$ estimate, there are two different noncommutative Calder\'on-Zygmund decompositions appearing in \cite{Par09} and \cite{CCP22}. In this paper, we choose the one in \cite{Par09} because it  works well within the \(TT^*\) argument, which is a standard tool in the study of rough operators.
In fact, since $\B^\del$ is essentially a rough operator, we will develop a $TT^*$ argument and geometric estimates within this noncommutative Calder\'on-Zygmund decomposition. More precisely, the good function is treated via Littlewood-Paley decompositions together with Fourier decay estimates for the kernel. We refer to Section \ref{s:314} for more details.
For the bad function,  we use a $TT^*$ argument together with a splitting of the double sum into  the near-piece and far-piece ranges. The near-piece range relies on a direct estimate of $T_j$ acting on the bad pieces, while the far-piece range is handled by the H\"older continuity of the composed kernel and the vanishing moments of the bad functions. The Calder\'on-Zygmund property $p_Qf_Qp_Q\lc\lam p_Q$ will play a key role in the $TT^*$ argument.
In both the near-piece and far-piece terms, we will encounter some very complicated geometric measure estimates that naturally arise from the off-diagonal terms, which do not appear in the commutative setting. For more details, we refer to Section \ref{s:315}.

\subsection*{Outline of the paper} Some preliminaries on noncommutative harmonic analysis including noncommutative $L_p$ spaces and quantum tori are presented in Section \ref{s:312}. We introduce two transference principles by transferring our main theorem to the operator-valued case $L_\infty(\T^d)\tens \M$ first and $L_\infty(\R^d)\tens\M$ second in Section \ref{s:31trans}.
Next the study of operator-valued Bochner--Riesz operator could be reduced to a class of singular oscillatory integral operator. By the noncommutative Calder\'on-Zygmund decomposition, we set up the proof of the weak type $(1,1)$ estimate for this singular oscillatory integral operator by reducing it to estimates for the good and bad functions. This will be carried out in Section \ref{s:31ope}. Finally, the proofs of the estimates for the good and bad functions are  presented in Section \ref{s:314} and Section \ref{s:315}, respectively.

\subsection*{Notation} Throughout the paper, the letter $C$ denotes a constant that may change from line to line.
$A\lc B$ means $A\leq CB$ for some constant $C$. By the notation $C_\eps$  we mean a constant depending on the parameter $\eps$. $A\approx B$ means that $A\lc B$ and $B\lc A$.
$\Z_+$ denotes the set of all nonnegative integers and $|x|$ denotes the $\ell_2$ norm. For $s\in\R_+$, $[s]$ denotes the integer part of $s$.
For any finite set $A$, we denote by $\card(A)$ or $\#(A)$ the number of elements in $A$.
Define $\mathcal{F}f$ (or $\hat{f}$) and $\mathcal{F}^{-1}f$ (or $\check{f}$) as the Fourier transform and the inverse Fourier transform
of $f$. For any dyadic cube $Q$, denote by $l(Q)$ the side length of the cube $Q$. For $t>0$, let $tQ$ be the cube with the same center as $Q$ and $l(tQ)=tl(Q)$.
In this paper, we only consider the critical index $\delta = \frac{d-1}{2}$ and  the dimension $d\ge2$.

\medskip

\section{Preliminaries on noncommutative harmonic analysis}\label{s:312}

In this section, we introduce some basic knowledge of noncommutative harmonic analysis including noncommutative $L_p$ spaces, maximal functions, square functions and many operator-valued inequalities which are useful in this paper.
All the material in this section is standard,  see \cite{JX07,PX03,Xu07} for more details.

\medskip

\subsection{Noncommutative $L_p$ spaces}
Let $\M$ be a semifinite von Neumann algebra with   a normal semifinite faithful (\emph{n.s.f.} for short) trace $\tau$. Denote by $\M_{+}$ the positive part of $\M$ and let $\mathcal{S}_{\M+}$  be the set of all $x\in\M_{+}$ whose support projection has finite trace. Let $\mathcal{S}_{\M}$ be the linear span of $\mathcal{S}_{\M+}$. For $1\le p<\infty$, $L_p(\M)$ is the completion of $\{x\in\mathcal{S}_{\M}: \tau(|x|^p)<\infty\}$ with respect to the norm $\|x\|_p = \tau(|x|^p)^{1/p}$. The weak $L_p$ space $L_{p,\infty}(\M)$ consists of all $\tau$-measurable operators $x$ with finite quasi-norm
\[
\|x\|_{p,\infty} = \sup_{\lam>0} \lam\, \tau(\chi_{(\lam,\infty)}(|x|))^{1/p}.
\]
An equivalent way to define the  $L_{p,\infty}(\M)$ quasi-norm is as follows:
\begin{align}\label{e:31weakm}
\sup_{\lambda>0}\lambda\inf_{e\in\mathcal{P}(\M)}
\Big\{\big(\tau(e^{\perp})\big)^{\frac{1}{p}}:\|exe\|_{\infty}\leq\lambda\Big\},
\end{align}
where $\mathcal{P}(\M)$ is the set of all projections in $\M$.
On $\R^d$ we consider the von Neumann algebra $\A = L_\infty(\R^d)\bar\otimes\M$ with trace $\vp(f) = \int_{\R^d} \tau(f(x))\,dx$. Then $L_p(\A)$ is isometric to the Bochner space $L_p(\R^d; L_p(\M))$.

\subsection{Vector-valued noncommutative $L_p$ spaces}
Let $(\Sigma,\mu)$ be a measure space. The column space $L_p(\M;L^c_2(\Sigma))$ consists of the operator-valued functions $f$ with finite norm for $p\geq1$ (quasi-norm for $0<p<1$)
\[
\|f\|_{L_p(\M;L_2^c(\Sigma))} = \Big\| \Big(\int_\Sigma f(\omega)^*f(\omega)\,d\mu(\omega)\Big)^{1/2} \Big\|_{L_p(\M)} < \infty.
\]
The following H\"older type inequality will be used repeatedly (see \cite[Proposition 1.1]{Mei07}).

\begin{lemma}\label{l:31holder-col}
Let $0<p,q,r\le\infty$ with $1/r = 1/p + 1/q$. For any $f\in L_p(\M;L_2^c(\Sigma))$ and $g\in L_q(\M;L_2^c(\Sigma))$,
\[
\Big\| \int_\Sigma f(\omega)^* g(\omega)\,d\mu(\omega) \Big\|_{r}
\le \Big\| \Big(\int_\Sigma |f(\omega)|^2\,d\mu(\omega)\Big)^{1/2} \Big\|_{p}
\Big\| \Big(\int_\Sigma |g(\omega)|^2\,d\mu(\omega)\Big)^{1/2} \Big\|_{q} .
\]
\end{lemma}
We also need the following convexity inequality (or the Cauchy-Schwarz type inequality) for operator-valued functions (see \cite[Page 9]{Mei07}). Let $(\mathfrak{m},\mu)$ be a measure space. Suppose that $f:\ \mathfrak{m}\rta\M$ is a weak-$*$ integrable function and $g:\mathfrak{m}\rta\mathbb C$ is an integrable function. Then
\Be\label{e:28conv}
\Big|\int_{\mathfrak{m}}f(x)g(x)d\mu(x)\Big|^2\leq\int_{\mathfrak{m}}|f(x)|^2d\mu(x)\int_{\mathfrak{m}}|g(x)|^2d\mu(x).
\Ee

\subsection{Conditional expectations and martingale differences}
Let $\Q$ be the set of all dyadic cubes in $\R^d$. For $k\in\Z$, $\Q_k$ denotes the family of dyadic cubes with side length $2^{-k}$. The conditional expectation $\E_k$ associated to the $\sigma$-algebra generated by $\Q_k$ acting on $L_p(\A)$ is given by
\[
\E_k f(x) = \sum_{Q\in\Q_k} f_Q \chi_Q(x),\qquad f_Q = \frac{1}{|Q|}\int_Q f(y)\,dy .
\]
Let $\{\A_k\}_{k \in \Z}$ stand for the corresponding
filtration, i.e.\ $\A_k = \mathsf{E}_k(\A)$. We write $f_k = \E_k f$. The martingale difference is $df_k = f_k - f_{k-1}$. The projections $\E_k$ are positive, trace preserving and satisfy $\E_k\E_l = \E_{k\wedge l}$.

\subsection{Quantum tori}
Suppose that $d\geq 2$, $\tet=(\tet_{k,j})_{1\leq k,j\leq d}$ is a real skew-symmetric $d\times d$ matrix.
The $d$-dimensional noncommutative torus $\A_\tet$ is a universal $C^*$-algebra generated by $d$ unitary operators $U_1,\cdots,U_d$ satisfying the following commutation relation:
\Bes
U_kU_j=e^{2\pi i\tet_{k,j}}U_jU_k,\quad 1\leq k,j\leq d.
\Ees

For a multi-index $k=(k_1,\cdots,k_d)\in\Z^d$ and $U=(U_1,\cdots,U_d)$, we define $U^k=U_1^{k_1}\cdots U_d^{k_d}$. We call
\Be\label{e:21quanpoly}
f=\sum_{k\in\Z^d}\alp_k U^k,\quad\text{with}\ \  \alp_k\in \C,
\Ee
a polynomial in $U$ if the sum in \eqref{e:21quanpoly} is finite.
Let $\mathcal{P}_\tet$ be the involution algebra of all such polynomials. Then $\mathcal{P}_\tet$ is dense in $\A_\tet$.
A well-known fact is that $\A_\tet$ admits a faithful tracial state $\tau$ such that $\tau(U_1^{k_1}\cdots U_d^{k_d})=1$ if $k=(k_1,\cdots,k_d)=\textbf{0}$ where $\textbf{0}=(0,\cdots,0)$ and $\tau(U^k)=\textbf{0}$ if $k\neq\textbf{0}$ (see e.g. \cite{Rie90}).
Hence for any polynomial $f$ of form  \eqref{e:21quanpoly}, one can define
$
\tau(f)=\alp_{\textbf{0}}.
$
Define $\T_\tet^d$ to be the weak $*$-closure of $\A_\tet$ in the GNS representation of $\tau$.
We call $\T_\tet^d$ the $d$-dimensional quantum torus associated to $\tet$. The state $\tau$ also extends to a normal faithful state on $\T_\tet^d$, which will be denoted again by $\tau$.
Notice that when $\tet=0$, $\A_\tet\backsimeq C(\T^d)$ and $\T_\tet^d\backsimeq L_\infty(\T^d)$. Thus quantum torus $\T_\tet^d$ is a deformation of classical torus $\T^d$.
Let $L_p(\T_\tet^d)$ be the noncommutative space associated with $(\T_\tet^d,\tau)$.

\medskip

\section{Two transference principles}\label{s:31trans}

In this section, we introduce two transference principles  for weak type $(1,1)$ boundedness. Before that some necessary definitions  are presented.
\begin{definition}
For $x\in\R^d$, a bounded function $m$ on $\R^d$ taking values in  $\mathbb{C}$ is called regulated at the point $x$ if
\Bes
\lim_{\eps\rta0}\fr{1}{\eps^d}\int_{|t|\leq\eps}(m(x-t)-m(x))dt=0.
\Ees
The function $m$ is called regulated if it is regulated at every $x\in\R^d$.
\end{definition}

The above definition of regulated function first appeared in \cite{deL65}.
Next we introduce the weak type $(1,1)$ boundedness of  Fourier multipliers on different noncommutative $L_1$ spaces.

\begin{definition}\label{d:31weakf}
(1) Given a bounded function $m:\Z^d\rta\mathbb{C}$, we say $\{m(n)\}_{n\in\Z^d}$ is a weak type $(1,1)$ bounded Fourier multiplier on $L_1(\T^d_\tet)$  if the operator $S_1$ defined by
\Bes
{S_1(f)}=\sum_{n\in\Z^d}m(n)\widehat{f}(n) U^n
\Ees
maps $L_1(\T^d_\tet)$ to $L_{1,\infty}(\T^d_\tet)$, where $f$ is given by $f=\sum_{n\in\Z^d}\widehat{f}(n)U^n$.

(2) Given a bounded function $m:\Z^d\rta\mathbb{C}$, we say $\{m(n)\}_{n\in\Z^d}$ is a weak type $(1,1)$ bounded Fourier multiplier on $L_1(L_\infty(\T^d)\tens\M)$  if the operator $S_2$ defined by
\Bes
{S_2(f)}(x)=\sum_{n\in\Z^d}m(n)\widehat{f}(n)e^{2\pi ix\cdot n}
\Ees
maps $L_1(L_\infty(\T^d)\tens\M)$ to $L_{1,\infty}(L_\infty(\T^d)\tens\M)$, where $f$ is a function: $\T^d\rta\M$ and $\widehat{f}(n)=\int_{\T^d}e^{-2\pi i n x}f(x)dx$.

(3) Given a bounded function $m:\R^d\rta\mathbb{C}$, we say $m$ is a weak type $(1,1)$ bounded Fourier multiplier on $L_1(L_\infty(\R^d)\tens\M)$  if the operator $S_3$ defined by
\Bes
{S_3(f)}(x)=\int_{\xi\in\R^d}m(\xi)\widehat{f}(\xi)e^{2\pi ix\cdot \xi}d\xi
\Ees
maps $L_1(L_\infty(\R^d)\tens\M)$ to $L_{1,\infty}(L_\infty(\R^d)\tens\M)$, where $f$ is a function: $\R^d\rta\M$ and $\widehat{f}(\xi)=\int_{\R^d}e^{-2\pi i \xi x}f(x)dx$.

\end{definition}

In the following, we introduce our first transference principle, which shows that the study of weak type $(1,1)$ boundedness on quantum tori can be reduced to that on usual tori.

\begin{theorem}\label{t:31transft}
Let $m:\Z^d\rta\mathbb{C}$ be a bounded function. Suppose $m$ is a Fourier multiplier which is of weak type $(1,1)$ on $L_1(L_\infty(\T^d)\tens\M)$. Then $m$ is a  weak type $(1,1)$ bounded Fourier multiplier on $L_1(\T_\tet^d)$.
\end{theorem}

To prove the above theorem, we need the following transference technique result which can be found in  \cite[Corollary 2.2]{CXY13}.
\begin{lemma}\label{l:31actiontori}
Let $\alpha$ be an action defined by
\begin{align}\label{tori3}
\alpha_{z} (U^{m}) = e^{2\pi i zm} U^{m } = e^{2\pi iz_1m_1} \cdots e^{2\pi iz_dm_d}U_1^{m_1} \cdots U_d^{m_d},
 \end{align}
where $z\in\T^d$ and  $U=(U_1,\cdots,U_d)$, the unitary operators given in the definition of quantum tori. Then the following conclusions hold:
\begin{enumerate}[\rm (i).]
	\item For all $x\in \mathbb{T}_{\theta}^{d}$, the map
	$z\mapsto \alpha _{z}x$ from $\mathbb{T}^{d}$ to $\mathbb{T}_{\theta}^{d}$ is continuous with respect
	to the w{*}-topology on $\mathbb{T}_{\theta}^{d}$.
	\item  For all $z\in \mathbb{T}^{d}$,
	$\alpha _{z}$ is an automorphism of $\mathbb{T}_{\theta}^{d}$.
    \item  For all $z\in \mathbb{T}^{d}$ and $1\le  p \leq \infty$, $\alpha _{z}$ is an isometry on $L_p (\mathbb{T}^d_{\theta})$.
	\item  For all $z\in \mathbb{T}^{d}$, $\tau =\tau\circ \alpha _{z}$.
\end{enumerate}
\end{lemma}

\medskip

\begin{proof}[Proof of Theorem \ref{t:31transft}]
Let $x\in L_{1}(\mathbb{T}_{\theta}^{d})$ and $S_1(x)$ be given in Definition \ref{d:31weakf}-(1).	 By the equivalent definition of $L_{1,\infty}$ in \eqref{e:31weakm}, it suffices to show for any
$\lambda >0$,  there exists
a projection ${e}\in \mathbb{T}_{\theta}^{d}$ such that
\begin{equation*}
\tau ({e}^{\bot })\lc
\lambda^{-1}\| x \| _{L_{1}(\mathbb{T}_{\theta}^{d})}\quad  \mbox{and} \quad  \|{e} S_1(x){e}\|_\infty \leq \lambda.
\end{equation*}
Fix $\lambda >0$ in the rest of the proof. Using the action $\alp$ and $x$, we define an operator-valued function $f\in L_{1}(\mathbb{T}^{d};L_{1}(\mathbb{T}_{\theta}^{d}))$ by setting $f(z)=\alpha_{z}x$ for $z\in\mathbb{T}^{d}$.
Then for each $z\in\mathbb{T}^{d}$,
\begin{equation}\label{eq:transf alpha_g An-1}
\alpha_{z}S_1(x)=S_2(\alpha_{z}x)=S_2(f)(z),
\end{equation}
where $S_2$ is an operator given in Definition \ref{d:31weakf}-(2).	
From  our assumption, $S_2$ is weak type $(1,1)$ bounded on $L_{1}(\mathbb{T}^{d};L_{1}(\mathbb{T}_{\theta}^{d}))$ (by setting $\M=\mathbb{T}_{\theta}^{d}$). Hence there exists a  projection
	$\tilde{e}\in L_{\infty }(\mathbb{T}^{d})\bar\otimes\mathbb{T}_{\theta}^{d}$  such that
	\begin{equation}\label{e:31tdweak}
	\int _{\mathbb{T}^{d}}\tau \big(\tilde{e}^{\bot}(z) \big)\,dz
	\lesssim \lam^{-1}
	{ \|  f \|  _{L_{1}(\mathbb{T}^{d};L_{1}(\mathbb{T}_{\theta}^{d}))}}   \quad \text{and}\quad
	\|\tilde{e}S_2(f)\tilde{e}\|_\infty\leq \lambda,
	\end{equation}
where we regard $\tilde{e}$ as a measurable operator-valued function on~$\mathbb{T}^{d}$. Combining with
	\eqref{eq:transf alpha_g An-1} and observing that each $\alpha _{z^{-1}}$ is a unitary $\tau$-preserving automorphism
	of $\mathbb{T}_{\theta}^{d}$ (see Lemma \ref{l:31actiontori} (ii) and (iv)), we get that for a.e. $z\in \mathbb{T}^{d}$
	\begin{equation}
	\|\big(\alpha _{z^{-1}}\tilde{e}(z) \bigr) S_1(x) \big(\alpha
	_{z^{-1}}\tilde{e}(z) \big)\|_{\mathbb{T}^{d}_\tet}\leq \lambda. \label{eq:a.s. lambda bdd}
	\end{equation}
For an arbitrary
	$\eps >0$, we may choose $z_{0}\in \mathbb{T}^{d}$ and a projection
	${e}\triangleq \alpha _{z_{0}^{-1}}\tilde{e}(z_{0})\in
	\mathbb{T}_{\theta}^{d}$ such that
	\begin{equation*}
	\tau ({e}^{\bot })\leq
	\inf _{z\in \mathbb{T}^{d}}\tau \bigl(\tilde{e}^{\bot }(z) \bigr)+
	\eps \quad \text{and}\quad  \|{e}S_1(x){e}\|_{\T_\tet^d}
	\leq \lambda.
	\end{equation*}
As a consequence, we obtain
\begin{align*}
	\tau ({e}^{\bot }) &\leq
	{|\mathbb{T}^{d}|^{-1}}
	\int _{\mathbb{T}^{d}}\tau \bigl(\tilde{e}^{
		\bot }(z) \bigr)\,dz+
	\eps \lesssim
	{\lambda^{-1}} \| f \| _{L_{1}(\mathbb{T}^{d};L_{1}(
		\mathbb{T}_{\theta}^{d}))}+
	\eps
	\\
	&={\lambda^{-1}}
	\int _{\mathbb{T}^{d}} \| \alpha _{z}x \|
	_{L_{1}(
		\mathbb{T}_{\theta}^{d})}\,dz+ \eps
	= {\lambda^{-1}} \| x \| _{L_{1}(
		\mathbb{T}_{\theta}^{d})}+
	\eps,
	\end{align*}
where in the second inequality we apply \eqref{e:31tdweak} and the last equality follows from Lemma \ref{l:31actiontori} (iii). Since $\eps$ is arbitrary, we finish the proof.
\end{proof}

Our next transference principle shows that the study of weak type $(1,1)$ boundedness on $L_1(L_\infty(\T^d)\tens\M)$ can be reduced to that on $L_1(L_\infty(\R^d)\tens\M)$.

\begin{theorem}\label{t:31transopeeuc}
Let $m:\R^d\rta\mathbb{C}$ be a bounded and regulated function. Suppose $m$ is a Fourier multiplier which is of weak type $(1,1)$ on $L_1(L_\infty(\R^d)\tens\M)$. Then $m$ is a  weak type $(1,1)$ Fourier multiplier on $L_1(L_\infty(\T^d)\tens\M)$.
\end{theorem}

\begin{proof}
Let $g\in L_{1}(L_\infty(\T^d)\tens\M)$ and $S_2(g)$ be given in Definition \ref{d:31weakf}-(2). To prove this theorem, it suffices to show for any
$\lambda >0$,  the following estimate holds:
\Bes
\int_{\T^d}\tau(|S_2(g)(x)|>\lam)dx\lc\lam^{-1}{\|g\|_{L_{1}(L_\infty(\T^d)\overline{\otimes}\mathcal M)}}.\Ees

By the density argument, it suffices to consider $g(x)=\sum_{n\in\Z^d}c_ne^{2\pi ix\cdot n}u_n$
where the sum over $n\in\Z^d$ is finite,	$c_n\in\mathbb{C}$ and $u_n\in \mathcal S_{\M}$. Without loss of generality, we may assume that $c_0=0$; otherwise just replace $g(x)$ by $g(x)-\int_{\T^d}g(z)dz$.
Fix $\lambda >0$ in our proof below.

For convenience, we identify $\mathbb{T}^d$ with the cube $\mathrm{Q}=[-1/2,1/2]^d$.
	Choose a positive Schwartz function $\phi$  supported in $2\mathrm Q$ and satisfying $\phi(x)=1$ if $x\in\mathrm{Q}$, $\phi(x)\leq1$ if $x\in2\mathrm{Q}$. For each integer $M>0$, set $\phi_{1/M}(x)=\phi(x/M)$. Recalling the definition of $S_3$ in Definition \ref{d:31weakf}-(3), a direct Fourier transform computation gives
$$\phi(x/M)S_2(g)(x)=S_3(g\phi_{1/M})(x)+J_{M}(x),$$
where
$$J_{M}(x)=\sum_{n\in\Z^d,n\neq0}c_ne^{2\pi ix\cdot n}u_n\int_{\R^d}\widehat{\phi}(y)e^{2\pi i\frac{y}{M}\cdot x}\Big(m(n)-m\big(n+\frac{y}{M}\big)\Big)dy.$$

We first claim that
\Be\label{e:31mlimi}
\lim_{M\rightarrow\infty}\|J_{M}(x)\|_{\infty}=0\  \mbox{uniformly\ in\ }\  x\in\R^d.
\Ee
The proof of \eqref{e:31mlimi} will be given later.

Note that for any $x\in M\mathrm{Q}$, since $S_2(g)(x)$ is a periodic function, we have
\Be\label{e:31S2S3}
S_2(g)(x)=\phi(x/M)S_2(g)(x)=S_3(g\phi_{1/M})(x)+J_{M}(x).
\Ee
Using again the periodicity of $S_2(g)$ and \eqref{e:31S2S3}, we obtain
\Be\label{e:31weaktoruo}
\begin{split}
\int_{\mathrm{Q}}\tau(|S_2(g)(x)|>\lam)dx&=\fr{1}{M^d}\int_{M\mathrm{Q}}\tau(|S_2(g)(x)|>\lam)dx\\
&\leq\fr{1}{M^d}\int_{\R^d}\tau(|S_3(g\phi_{1/M})(x)|>{\lam}/{2})dx\\
&\quad+\fr{1}{M^d}\int_{M\mathrm{Q}}\tau(|J_M(x)|>{\lam}/{2})dx.
\end{split}
\Ee

By our assumption that $S_3$ is of weak type $(1,1)$ on $L_1(L_\infty(\R^d)\tens\M)$, we get
\Be\label{e:31weaktoru}
\begin{split}
\int_{\R^d}\tau(|S_3(g\phi_{1/M})(x)|>{\lam}/{2})dx&\lc\lam^{-1}{\|g\phi_{1/M}\|_{L_{1}(L_\infty(\R^d)\overline{\otimes}\mathcal M)}}\\
&\lc\lam^{-1}M^d{\|g\|_{L_{1}(L_\infty(\mathrm{Q})\overline{\otimes}\mathcal M)}},
\end{split}\Ee
where the last inequality follows from $\supp(\phi_{1/M})\subset 2M\mathrm{Q}$ and $|\phi(x)|\leq1$. On the other hand, using \eqref{e:31mlimi}, we could choose $M$ sufficiently large such that $\|J_M(x)\|_{\M}\leq\lam/2$ uniformly in $x\in\R^d$. Combining this fact, \eqref{e:31weaktoru} and \eqref{e:31weaktoruo}, we obtain
\Be
\int_{\mathrm{Q}}\tau(|S_2(g)(x)|>\lam)dx\lc\lam^{-1}{\|g\|_{L_{1}(L_\infty(\mathrm{Q})\overline{\otimes}\mathcal M)}}.\Ee
Hence we prove that $S_2$ is of weak type $(1,1)$ on $L_1(L_{\infty}(\T^d)\overline{\otimes}\mathcal M)$.

It remains to show \eqref{e:31mlimi}.
	Note that for any  $M$,
	\begin{align*}
		\quad	\|J_{M}\|_{\infty}
		 \leq\sum_{n\in\Z^d,n\neq0}|c_n|\|u_n\|_{\M}\int_{\R^d}|\widehat{\phi}(y)|\Big|m(n)-m\big(n+\frac{y}{M}\big)\Big|dy.
	\end{align*}
	Since the sum over $n\in\Z^d$ is finite, it is enough to show for any fixed $n\in\Z^d$,
	$$\lim_{M\rightarrow\infty}\int_{\R^d}|\widehat{\phi}(y)|\Big|m(n)-m\big(n+\frac{y}{M}\big)\Big|dy=0.$$
Recall that $m$ is bounded and $\widehat{\phi}$ is a Schwartz function. Then for any $\eps>0$, there exists a positive constant $N$ such that	
$$\int_{|y|\geq N}|\widehat{\phi}(y)|\Big|m(n)-m\big(n+\frac{y}{M}\big)\Big|dy\leq 2\|m\|_{\infty}\int_{|y|\geq N}|\widehat{\phi}(y)|dy\leq \eps.$$
On the other hand, we obtain
\Bes
\int_{|y|\leq N}|\widehat{\phi}(y)|\Big|m(n)-m\big(n+\frac{y}{M})\Big|dy
\leq\|\widehat{\phi}\|_\infty{M^d}\int_{|y|\leq \fr{N}{M}}|m(n)-m(n+y)|dy
\Ees
which converges to $0$ as $M\rta+\infty$ because $m$ is a regulated function.
Since $\eps>0$ is arbitrarily chosen, we complete the proof of \eqref{e:31mlimi}.
\end{proof}
Set $\A = L_\infty(\R^d)\bar\otimes\M$ and define the n.s.f.\ trace $\vp$ on $\A$ by
$
\vp(f)=\int_{\R^d}\tau(f(x))\,dx .
$
The operator-valued Bochner--Riesz operator of order $\delta>0$ is defined on the Fourier side by
\Be\label{e:31bochop}
\widehat{B^\delta f}(\xi) = (1-|\xi|^2)_+^\delta \hat f(\xi),\qquad f\in L_2(\A)\cap L_1(\A).
\Ee
Note that the Fourier multiplier $m(\xi)= (1-|\xi|^2)_+^\delta $ is a bounded and regulated function. Then by Theorems \ref{t:31transft} and \ref{t:31transopeeuc}, to prove our main Theorem \ref{t:31main}, it suffices to establish the following result.

\begin{theorem}\label{t:31mainop}
Let $B^\del$ be defined in \eqref{e:31bochop} with $\del=\fr{d-1}{2}$. Then $B^\del$ is of weak type $(1,1)$ on $L_1(\A)$. That is, for all $f\in L_1(\A)$ and $\lam>0$, we have
\[
\lam\,\vp\bigl(\chi_{(\lam,\infty)}(|B^\delta f|)\bigr) \lc\|f\|_{L_1(\A)} .
\]
\end{theorem}

For the proof of Theorem \ref{c:31cor}, we have the following remark.
\begin{remark}
By a scaling argument, Theorem \ref{t:31mainop} also holds  uniformly in $R>0$ for $B^\del_R$  with Fourier multiplier $m(\xi)=(1-|\xi/R|^2)_+^\delta$. Then using Theorems \ref{t:31transft} and \ref{t:31transopeeuc}, we get a similar result on quantum tori. That is, $B^\del_R$ is of weak type $(1,1)$ on $L_1(\T_\tet^d)$ with bound independent of $R$.  Now, by the standard limiting argument, we could show that for $f\in L_1(\T^d_\tet)$, $\B^\del_R(f)$ converges to $f$ as $R\rta +\infty$ in measure which finishes the proof of Theorem \ref{c:31cor}. We omit the details here.
\end{remark}

\medskip

\section{Proof of operator-valued Bochner--Riesz operator: Theorem \ref{t:31mainop}}\label{s:31ope}

In this section, we  prove Theorem \ref{t:31mainop}. We first reduce it to a singular oscillatory integral operator.

\subsection{A reduction to singular oscillatory integral operator}

In this case $\del=\frac{d-1}{2}$, the Bochner--Riesz kernel $K_\delta(x) = \mathcal{F}^{-1}[(1-|\cdot|^2)_+^\delta](x)$ can be decomposed as follows.

\begin{lemma}[see Christ \cite{Chr88}  Section 3]\label{l:31christ}
We have the following decomposition
\[
K_\delta(x) = \sum_{j=0}^\infty K_j(x) + L(x),
\]
where
\[
K_j(x) =\psi(x)|x|^{-d} \phi(2^{-j} x),\qquad \psi(x) = c_d {\cos\big(2\pi(|x|-{(d-1)}/{4})\big)},
\]
$\phi$ is a suitable smooth radial function  supported in $\{2^{-2}\le|x|\le 2\}$,  $c_d$ is a constant,
and the remainder $L(x)$ is an integrable function satisfying $$|L(x)|\lc(1+|x|)^{-(d+1)}.$$
\end{lemma}
Consequently, the operator with kernel $L(x)$ maps $L_1(\A)$ to $L_1(\A)$, which of course implies its weak type $(1,1)$ boundedness. Therefore it suffices to treat the operator of convolution with $\sum_{j \geq 0}K_j$.
For convenience,  set $T = \sum_{j=0}^\infty T_j$ where
$$T_j f(x) = \int_{\R^d} K_j(x-y)f(y)\,dy$$
for each $j\ge 0$.
In the rest of this paper, we will show that $T$ is  weak type $(1,1)$ bounded on $L_1(A)$, that is for all $f\in L_1(\A)$ and $\lam>0$, the following estimate
\Be\label{e:31mainweak}
\lam\,\vp\bigl(\chi_{(\lam,\infty)}(|T f|)\bigr) \lc \|f\|_{L_1(\A)}
\Ee
holds.
Since every \(f\in L_1(\A)\) can be expressed as a linear combination of four positive elements of \(L_1(\A)\), it suffices to consider positive \(f\in L_1(\A)\).
Moreover, by a density argument, it suffices to consider functions \(f\) in the following dense subclass of \(L_1(\A)_+\):
\[
\A_{c,+}=\{f\in L_1(\A)_+:\ \operatorname{supp} f \text{ is compact}\}.
\]
Here \(\operatorname{supp} f\) denotes the support of \(f\), viewed as an operator-valued function on \(\R^d\). That is,
\[
\operatorname{supp} f=\overline{\{x\in\R^d:\ f(x)\neq0\}}.
\]
In the following, we introduce the noncommutative Calder\'on-Zygmund decomposition.

\subsection{Noncommutative Calder\'on-Zygmund decomposition}\label{s:31s42}
Define $f_k=\mathsf{E}_kf$ for each $k\in\Z$. Then $\{f_k\}_{k\in\Z}$ forms a positive dyadic martingale in $L_1(\A)$. Applying the well-known Cuculescu construction from \cite[Lemma 3.1]{Par09} at level $\lam$ yields the following result.

\begin{lemma}\label{l:28cucu}
There exists a decreasing sequence \(\{q_k\}_{k\in\Z}\) of projections, depending on \(f\) and \(\lambda\), such that:
\begin{enumerate}[\rm (i).]
\item \(q_k\) commutes with \(q_{k-1}f_kq_{k-1}\) for all \(k\in\Z\);
\item \(q_k\in\A_k\) for every \(k\in\Z\), and \(q_kf_kq_k\le \lambda q_k\);
\item Let \(q=\bigwedge_{k\in\Z}q_k\), and let \(1_\A\) be the identity on \(\A\). Then
\[
\vp(1_\A-q)\le \lambda^{-1}\|f\|_{L_1(\A)};
\]
\item The projections $q_k$ can be expressed as follows: for some negative integer $m\in\Z$,
\Bes
q_k=\begin{cases}
1_\A\ \ &\qquad\text{if\ $k<m$,} \\
\chi_{(0,\lam  ]}(f_k)\ &\qquad\text{if\ $k=m$},\\
\chi_{(0,\lam  ]}(q_{k-1}f_kq_{k-1}) &\qquad\text{if\ $k>m$}.
\end{cases}
\Ees
\end{enumerate}
\end{lemma}

As in \cite{Par09}, we next give an alternative description of the projections \(q_k\). This description will be useful when estimating quantities involving \(q_k\). Indeed, for every \(k\in\Z\), we can write
\[
q_k=\sum_{Q\in\Q_k}\xi_Q\chi_Q,
\]
where each \(\xi_Q\) is a projection in \(\M\) such that:
\begin{enumerate}[\rm (i).]
\item We have the following explicit formula, where \(\widehat{Q}\) denotes the parent dyadic cube of \(Q\):
\[
\xi_Q=\begin{cases}
1_\M, & \text{if } k<m, \\[2pt]
\chi_{(0,\lambda]}(f_Q), & \text{if } k=m, \\[2pt]
\chi_{(0,\lambda]}(\xi_{\widehat{Q}}f_Q\xi_{\widehat{Q}}), & \text{if } k>m.
\end{cases}
\]
\item \(\xi_Q\in\mathcal{P}(\M)\) and \(\xi_Q\le\xi_{\widehat{Q}}\);
\item \(\xi_Q\) commutes with \(\xi_{\widehat{Q}}f_Q\xi_{\widehat{Q}}\), and \(\xi_Qf_Q\xi_Q\le\lambda\xi_Q\).
\end{enumerate}

Set \(p_k=q_{k-1}-q_k\). Using the above explicit representation, we obtain
\[
p_k=\sum_{Q\in\Q_k}(\xi_{\widehat{Q}}-\xi_Q)\chi_Q=:\sum_{Q\in\Q_k}p_Q\chi_Q,
\]
with \(p_Q=\xi_{\widehat{Q}}-\xi_Q\).
It is easily verified that the \(p_k\)'s are pairwise disjoint and satisfy \(\sum_{k\in\Z}p_k=1_\A-q\).

The associated good and bad functions of $f$ are then defined by
\Bes
f=g+b,\quad g=\sum_{i,j\in\widehat{\Z}}p_if_{i\vee j}p_j,\quad b=\sum_{i,j\in\widehat{\Z}}p_i(f-f_{i\vee j})p_j,
\Ees
where we put $p_\infty=q$, $\widehat{\Z}=\Z\cup\{\infty\}$ and $i\vee j=\max (i,j)$. When $i$ or $j$ equals $\infty$, $i\vee j$ is understood to be $\infty$ and $f_\infty:=f$. The above construction is due to Parcet \cite{Par09}.

In the following we first state a lemma providing a projection in $\A$ allowing us to reduce the problem to the case where all operators are restricted to that projection.
\begin{lemma}\label{l:19excep}
There exists a projection $\zet$ satisfying the properties:
\begin{enumerate}[\rm (i).]
\item $\lam\vphi(1_\A-\zet)\lc  \|f\|_{L_1(\A)}$;
\item For any $Q\in\Q$ and any $x\in (2^{101}+1)Q$, we have $\zet(x)\leq1_\M-\xi_{\widehat{Q}}+\xi_{Q}$ and $\zet(x)\leq\xi_{Q}$.
\end{enumerate}
\end{lemma}

The proof of this lemma is a straightforward modification of the argument in \cite[Lemma 4.2]{Par09}. The particular constant \(2^{101}+1\) is not crucial.

\subsection{Properties of the good and bad functions}
We decompose \(g\) into its diagonal and off-diagonal terms:
\begin{equation}\label{e:31gooddecom}
g_d=qfq+\sum_{k\in\Z}p_kf_kp_k,\qquad
g_\off=\sum_{\substack{i,j\in\Z\\ i\ne j}}p_if_{i\vee j}p_j
+qf(1_\A-q)+(1_\A-q)fq.
\end{equation}

Below we recall some properties of the diagonal term \(g_d\), the off-diagonal term \(g_\off\), and related estimates that were proved in \cite{Par09}.
\begin{lemma}\label{l:19gbasic}
For the diagonal term \(g_d\), we have the following basic estimates:
\begin{equation}\label{e:19gdes}
\|g_d\|_{L_1(\A)}\lc\|f\|_{L_1(\A)},\qquad
\|g_d\|_{L_\infty(\A)}\lc\lambda.
\end{equation}

Let \(df_s\) be the martingale difference. For the off-diagonal term \(g_\off\), we rewrite it as follows:
\[
g_\off=\sum_{s\ge1}\sum_{k\in\Z}
\bigl(p_k\,df_{k+s}\,q_{k+s-1}+q_{k+s-1}\,df_{k+s}\,p_k\bigr)
=:\sum_{s\ge1}\sum_{k\in\Z}g_{k,s}
=:\sum_{s\ge1}g_{(s)}.
\]
The martingale difference sequence of \(g_{(s)}\) satisfies
\((d g_{(s)})_{k+s}=g_{k,s}\).
Meanwhile, we have the estimate
\begin{equation}\label{e:19gbasic}
\sup_{s\ge1}\|g_{(s)}\|_{L_2(\A)}^2
=\sup_{s\ge1}\sum_{k\in\Z}\|g_{k,s}\|_{L_2(\A)}^2
\lc\lambda\|f\|_{L_1(\A)}.
\end{equation}
\end{lemma}

For the bad function, let us introduce the following expression:
\[
b=\sum_{k\in\Z}p_k(f-f_k)p_k
+\sum_{s\ge1}\sum_{k\in\Z}
\bigl(p_k(f-f_{k+s})p_{k+s}+p_{k+s}(f-f_{k+s})p_k\bigr)
=:\sum_{s\ge0}\sum_{k\in\Z}b_{k,s},
\]
where
\begin{equation}\label{e:19db}
b_{k,0}=p_k(f-f_k)p_k,\qquad
b_{k,s}=p_k(f-f_{k+s})p_{k+s}+p_{k+s}(f-f_{k+s})p_k.
\end{equation}

We now list some basic properties of the bad functions that will be helpful in our later proof.
\begin{lemma}\label{l:19bbasic}
Let \(b_{k,s}\) be defined as in \eqref{e:19db}. For any fixed \(s\ge0\), the following results hold:
\begin{enumerate}[\rm (i).]
\item An $L_1$ bound estimate:
$
\sum_{k\in\Z}\|b_{k,s}\|_{L_1(\A)}\lc\|f\|_{L_1(\A)};
$
\item The cancellation condition: for all $k\in\Z$ and every $Q\in\Q_{k+s}$,
$\int_{Q}b_{k,s}(y)dy=0$.
\end{enumerate}
\end{lemma}

The proof of Lemma \ref{l:19bbasic} can be found in \cite{Cad18,Par09}.

\medskip

From the linearity of $T$ and the above noncommutative Calder\'on-Zygmund decomposition, we deduce
\Bes
{\vphi}(|Tf|>\lam)\leq{\vphi}(|Tg|>\lam/2)+{\vphi}(|Tb|>\lam/2).
\Ees

The proofs of the estimates related to the good and bad functions will be given separately in Section \ref{s:314} and Section \ref{s:315}, respectively.

\section{Proof of the estimates related to the good functions}\label{s:314}

For technical reasons, we first make a finite microlocal decomposition of the kernel $K_j$. That is, we fix a finite partition of unity $\{\eta_\nu\}_{\nu\in\Theta}$ on the unit sphere $\S^{d-1}$, with each $\eta_\nu$ having very small support such that $\sum_{\nu\in\Theta}\eta_{\nu}(\theta)=1$ for all $\theta\in\S^{d-1}$. The sum over $\nu$ is finite and we denote $\card(\Theta)=c_1$. Next we define
\[
T_j^\nu h(x) = K_j^\nu * h(x),
\]
where
\[
K_j^\nu(x) = K_j(x)\eta_\nu(x/|x|).
\]
 Set $T^\nu=\sum_{j \geq 0}T_j^\nu$.
Then $T=\sum_{\nu\in\Theta}T^\nu$. Therefore we get
\[
{\vphi}(|Tg|>\lam/2)\leq \sum_{\nu\in\Theta}{\vphi}\big(|T^\nu g|>\lam/(2c_1)\big).
\]
Thus, we only need to show that for each $\nu\in\Theta$, ${\vphi}\big(|T^{\nu}g|>\lam/(2c_1)\big)$ satisfies the desired estimate. In the remaining part of this subsection, we fix a $\nu\in\Theta$.

To proceed further, we need the following kernel boundedness estimate proved by M. Christ in \cite[Lemma 3.1]{Chr88}.

\begin{lemma}\label{l:31kernel1}
Let $\tilde{K}_j^\nu(x)={K}_j^\nu(-x)$. The following estimate holds
\[
|{K_j^\nu * \tilde{K}_j^\nu(x)}| \lesssim 2^{-jd}(1+|x|)^{-\delta},
\]
where $\delta=(d-1)/2$.
\end{lemma}

Since the good function can be decomposed into the diagonal term  $g_d$ and the off-diagonal term  $g_\off$, the proofs of $g_d$ and $g_\off$ will be different, as we shall see below.

\medskip

\subsection{Estimate of the diagonal term  $g_d$}

We first show that $T^\nu$ is bounded on $L_2(\A)$ by utilizing Lemma \ref{l:31kernel1}. In fact, for any $\xi\in\R^d$,
\Be\label{e:31Knuj}
\begin{split}
|\widehat{K}^\nu_j(\xi)|^2&=|\F({K}^\nu_j*\tilde{K}^\nu_j)(\xi)|\leq\int |{K}^\nu_j*\tilde{K}^\nu_j(x)|dx\\
&\lc\int_{|x|\lc2^j}2^{-jd}(1+|x|)^{-\delta}dx\lc 2^{-j\delta},
\end{split}
\Ee
where the third inequality follows from the support of ${K}^\nu_j*\tilde{K}^\nu_j$ and Lemma \ref{l:31kernel1}.
Next using the Plancherel theorem twice, we get
\[
\begin{aligned}
\Norm{K_j^\nu * h}_{L_2(\A)}^2
\leq \|\widehat{K}^\nu_j\|_{\infty}^2\Norm{h}_{L_2(\A)}^2 \lc2^{-j\del} \Norm{h}_{L_2(\A)}^2.
\end{aligned}
\]
Now applying the triangle inequality, we obtain
\Be\label{e:31L2}
\|T^\nu h \|_{L_2(\A)}\leq\sum_{j \geq 0}\|T_j^\nu h \|_{L_2(\A)}
\lc\sum_{j \geq 0}2^{-j\delta/2}\|h\|_{L_2(\A)}\lc\|h\|_{L_2(\A)}.
\Ee
Thus, we have proved that $T^\nu$ is bounded on $L_2(\A)$ .

Based on the $L_2(\A)$ boundedness of $T^\nu$, we can prove the estimate for the diagonal term $g_d$ as follows:
\Bes
\begin{split}
\vphi \big(|T^\nu g_d|>\lam/(4c_1)\big)&\lc\lam^{-2}\|T^\nu g_d\|^2_{L_2(\A )}
\lc\lam^{-2}\|g_d\|^2_{L_2(\A)}\\
&\lc\lam^{-2}\|g_d\|_{L_1(\A)}\|g_d\|_{L_\infty(\A)}\lc\lam^{-1}\|f\|_{L_1(\A)},
\end{split}
\Ees
where in the first inequality we use  the Chebyshev  inequality, the second inequality follows from the $L_2(\A)$ boundedness of $T^\nu$, while in the third and fourth inequalities we apply \eqref{e:19gdes}. Hence, for the diagonal term $g_d$, we get the desired estimate.

\subsection{Estimate of the off-diagonal term  $g_\off$}

In the remaining part of this subsection, we focus on the estimate for $g_\off$. We first utilize Lemma \ref{l:19excep} to reduce the proof to the operator $\zet T^\nu g_\off \zet$.
Indeed,
\Bes
T^\nu g_\off=(1_\A-\zeta)T^\nu g_\off(1_\A-\zet)+\zet T^\nu g_\off(1_\A-\zet)+(1_\A-\zet)T^\nu g_\off\zet+\zet T^\nu g_\off\zet.
\Ees
By the property (i) in Lemma \ref{l:19excep}, we obtain
\Bes
\begin{split}
\vphi \big(|T^\nu g_\off|>\lam/(4c_1)\big)&\lc\vphi(1_\A-\zet)+\vphi \big(|\zet T^\nu g_\off\zet|>\lam/(8c_1)\big)\\
&\lc\lam^{-1}\|f\|_{L_1(\A)}+\vphi \big(|\zet T^\nu g_\off\zet|>\lam/(8c_1)\big).
\end{split}
\Ees
Thus our goal is to show the estimate
\Be\label{e:19goffg}
\vphi\big(|\zet T^\nu g_\off \zet|>\lam/(8c_1)\big)\lc\lam^{-1}\|f\|_{L_1(\A)}.
\Ee

Applying the expression of $g_\off$ in Lemma \ref{l:19gbasic} and the formula $f=\sum_{n\in\Z}df_n$, we can further write
\Bes
\begin{split}
\zet T^\nu g_\off\zet&=\zet\sum_{s\geq1}\sum_{j \geq 0} T_j^\nu g_{(s)}\zet=\zet\sum_{s\geq1}\sum_{j \geq 0}\sum_{n\in\Z} T_j^\nu d\big(g_{(s)}\big)_{n-j+s}\zet\\
&=\zet\sum_{s\geq1}\sum_{j \geq 0}\sum_{n\in\Z}   T_j^\nu g_{n-j,s}\zet
=\zet\sum_{s\geq1}\sum_{n\geq100}\sum_{j \geq 0}   T_j^\nu g_{n-j,s}\zet
\end{split}
\Ees
where in the last equality we apply the following observation: if $\zet(x)\neq0$ and $n<100$,
we get that for all $s\geq1$,
\Be\label{e:28exec}
\begin{split}
\zet(x)T^\nu_jg_{n-j,s}(x)\zet(x)
&=\sum_{Q\in\Q_{n-j}}\zet(x)\chi_{((2^{101}+1)Q)^c}(x)\int_Q K^\nu_j(x-y)\\
&\quad\times\big[ p_{Q}(df_{n-j+s}) q_{n-j+s-1}
+ q_{n-j+s-1}(df_{n-j+s}) p_{Q}\big]dy\,\zet(x)\\
&=0
\end{split}
\Ee
here in the first equality we use $\zet(x) p_{Q}=0$ if $x\in (2^{101}+1)Q $ by the property (ii) of $\zet$ in Lemma \ref{l:19excep},
and the second equality follows from the fact that $x\in ((2^{101}+1)Q )^c$ and $y\in Q$ implies $|x-y|\geq2^{100+j-n}$,
which contradicts the support of $K^\nu_j$ when $n<100$.

By  the Chebyshev  inequality, the triangle inequality and the fact $\zet$ is a projection in $\A$, we then obtain
\Bes
\begin{split}
\vphi \big(|\zet T^\nu g_\off \zet|>\lam/(8c_1)\big)&\lc\lam^{-2}\|\zet T^\nu g_\off \zet\|_{L_2(\A )}^2\\
&\lc\lam^{-2}\bigg(\sum_{s\geq1}\sum_{n\geq100}\Big\|\sum_{j \geq 0} T^\nu_j g_{n-j,s}\Big\|_{L_2(\A)}\bigg)^2.
\end{split}
\Ees
Hence, to finish the proof of the off-diagonal term $g_\off$, it is sufficient to show the following proposition.

\begin{proposition}\label{l:31good}
For any $s \geq 1$ and $n \geq 100$, there exists a positive constant $\alpha_1$ such that
\begin{equation}
\Big\|{\sum_{j \geq 0} T_j^\nu g_{n-j,s}}\Big\|_{L_2(\A)}^2 \lesssim 2^{-(n+s)\alpha_1} \lambda \Norm{f}_{L^1(\A)}.
\end{equation}
\end{proposition}

\medskip
\subsection{Proof of Proposition \ref{l:31good}}

Let us first introduce the Littlewood-Paley decomposition. 
Let $\Phi$ be a radial positive
 $C^\infty$ function such that $\Phi(\xi)=1$ for $|\xi|\leq 1$, $\Phi(\xi)=0$ for $|\xi|\geq 2$,
and $\Phi(\xi)\leq1$ for all $\xi\in\R^d$. Define $\beta_k(\xi)=\Phi(2^k\xi)-\Phi(2^{k+1}\xi)$;
then $\beta_k$ is supported in $\{\xi:2^{-k-1}\leq|\xi|\leq 2^{-k+1}\}$. Choose $\tilde\beta$ to be a radial positive
 $C^\infty$ function such that $\tilde{\beta}(\xi)=1$ for $\frac12\leq|\xi|\leq 2$, $\tilde{\beta}$ is supported in $\{\xi:\frac14\leq|\xi|\leq4\}$,
and $\tilde{\beta}(\xi)\leq1$ for all $\xi\in\R^d$. Set $\tilde{\beta}_k(\xi)=\tilde{\beta}(2^k\xi)$; then it is easy to see $\beta_k=\tilde{\beta}_k\beta_k$. Define the convolution operators $\Delta_k$ and $\tilde{\Delta}_k$
with the Fourier multipliers $\beta_k$ and $\tilde{\beta}_k$, respectively.
 That is,
$$\widehat{{\Delta}_kf}(\xi)=\beta_k(\xi)\hat{f}(\xi),\quad \ \widehat{\tilde{\Delta}_kf}(\xi)=\tilde{\beta}_k(\xi)\hat{f}(\xi).$$
Then by the constructions of $\beta_k$ and $\tilde{\beta}_k$, we have
$\Delta_k=\tilde{\Delta}_k\Delta_k$ and the identity ${\rm{I}}=\sum\limits_{k\in\mathbb{Z}}\Delta_k.$

The author in \cite[Page 1468, Line -8]{Lai24} used the Schur lemma to show that for all $k-j+n+s>0$,
\Be\label{e:28gnj}
\|\Delta_{k}g_{n-j,s}\|_{L_2(\A)}\lc2^{j-n-s-k}\|g_{n-j,s}\|_{L_2(\A)}.
\Ee

By inserting the Littlewood-Paley decomposition, we write
\Be\label{e:28L2}
\begin{split}
\Big\|\sum_{j \geq 0}  T^\nu_jg_{n-j,s}\Big\|^2_{L_2(\A)}&=\Big\|\sum_{k\in\Z}\sum_{j \geq 0}  T^\nu_j\Delta_{k}g_{n-j,s}\Big\|^2_{L_2(\A)}\\
&\lc\sum_{k\in\Z}\tau\int\Big|\sum_{j \geq 0}  \widehat{K^\nu_j}(\xi)\beta_{k}(\xi)\widehat{g_{n-j,s}}(\xi)\Big|^2d\xi
\end{split}
\Ee
where the second inequality follows from the Plancherel  theorem and the fact that the supports of $\beta_{k}$ are essentially disjoint (in fact $\supp(\beta_{k})$ only intersects $\supp(\beta_{k+l})$ for $l=-1,0,1$).
Next we split the sum over $j \geq 0$ into two cases: $0\leq j\leq c_{k,n,s}$ and $j\geq c_{k,n,s,1}$, where $c_{k,n,s}=:[k+(n+s)\epsilon]-1$, $c_{k,n,s,1}=:\max\{[k+(n+s)\epsilon],0\}$ and $\epsilon\in(0,1)$ is a fixed constant.

First let us consider the sum over $0\leq j \leq c_{k,n,s}$. By the convex inequality for operator-valued functions \eqref{e:28conv} and the estimate of $K^\nu_j$ in \eqref{e:31Knuj}, we get
\Be\label{e:28Kj11}
\begin{split}
&\quad\sum_{k\in\Z}\tau\int\Big|\sum_{j=0}^{c_{k,n,s}}  \widehat{K^\nu_j}(\xi)\beta_{k}(\xi)\widehat{g_{n-j,s}}(\xi)\Big|^2d\xi\\
&\lc\sum_{k\in\Z}\tau\int\Big(\sum_{j \geq 0}|\widehat{K^\nu_j}(\xi)|^2\Big)  \Big(\sum_{j=0}^{c_{k,n,s}} \big|\beta_{k}(\xi)\widehat{g_{n-j,s}}(\xi)\big|^2\Big)d\xi\\
&\lc\sum_{k\in\Z}\sum_{j=0}^{c_{k,n,s}}\|\check{\beta}_{k}*g_{n-j,s}\|_{L_2(\A)}^2,
\end{split}
\Ee
where in the last inequality we also use the Plancherel theorem. Notice that $j\leq c_{k,n,s}$ implies that $k-j+n+s>0$ since $\epsilon\in(0,1)$.
Applying \eqref{e:28gnj}, changing the order of the sums over $k$ and $j$, together with Lemma \ref{l:19gbasic}, the above estimate is further bounded by
\Be\label{e:28jleqk}
\begin{split}
\sum_{j \geq 0}\sum_{j-k<(n+s)\epsilon} 2^{2(j-k-n-s)}\|g_{n-j,s}\|_{L_2(\A)}^2
\lc 2^{-(n+s)2(1-\epsilon)}\lam\|f\|_{L_1(\A)}.
\end{split}
\Ee

In the following we consider the sum over $j\geq c_{k,n,s,1}$.  By the convex inequality for operator-valued functions \eqref{e:28conv}, we obtain
\Be\label{e:31kjgeq}
\begin{split}
&\quad\sum_{k\in\Z}\tau\int\Big|\sum_{j\geq c_{k,n,s,1}}  \widehat{K^\nu_j}(\xi)\beta_{k}(\xi)\widehat{g_{n-j,s}}(\xi)\Big|^2d\xi\\
&\leq\sum_{k\in\Z}\tau\int\Big(\sum_{j\geq c_{k,n,s,1} }|\widehat{K^\nu_j}({\xi})\beta_k(\xi)|^2\Big)\Big(\sum_{j \geq 0}  |\tilde{\beta}_{k}(\xi)\widehat{g_{n-j,s}}(\xi)|^2\Big)d\xi.
\end{split}
\Ee

We first give an explicit estimate of $\widehat{K^\nu_j}$. Recall that the microlocal kernel $K_j^\nu(x)=\psi(x)|x|^{-d} \phi(2^{-j} x)\eta_\nu(\frac{x}{|x|})$ where $\psi(x)$, $\phi(x)$ are radial functions.
For a fixed $\xi\in\R^d$, we  rewrite
\[
\begin{aligned}
\widehat{K^\nu_j}(\xi) = \int_{\mathbb{R}^d} e^{-2\pi i \xi\cdot x} \eta_\nu(\frac{x}{|x|}) \psi(x) \frac{\phi(2^{-j} x)}{|x|^d} dx = \int_{\mathbb{R}^d} e^{-2\pi i |\xi| x_1} \eta_{\tilde\nu}(x) \psi(x) \frac{\phi(2^{-j} x)}{|x|^d} dx,
\end{aligned}
\]
where we make a linear change of variables $x \mapsto l(x)$, $l$ is a rotation such that its transpose satisfies
\[
l^*\big({\xi}/{|\xi|}\big) = (1,0,0,\dots,0),
\]
and $\eta_{\tilde\nu}(x)=\eta_{\nu}(l(x)/|x|)$.

Notice that $\eta_{\tilde\nu}(x)\psi(x)\varphi(2^{-j} x)/|x|^{d}$ is a compactly supported smooth function. Integrating by parts with respect to $x_1$, we get
\Bes
\begin{split}
|{\widehat{K_j^\nu}(\xi)}|&=\Big|\int e^{-2\pi i|\xi|x_1} \frac{1}{2\pi i |\xi|}\pari_{x_1}\big[\eta_{\tilde\nu}(x)\psi(x){\phi(2^{-j} x)}{|x|^{-d}}\big] dx\Big|\\
&\lesssim |\xi|^{-1} \int_{|x|\approx 2^{j}} |\pari_{x_1}\big[\eta_{\tilde{\nu}}(x)\psi(x){\phi(2^{-j} x)}{|x|^{-d}}\big]| dx \lesssim |\xi|^{-1},
\end{split}
\Ees
where the last inequality follows from an elementary calculation and  $j \geq 0$. Combining the above estimate and \eqref{e:31Knuj}, we  deduce
\[
|{ \widehat{K_j^\nu}(\xi)}|\lc\min\{2^{-j\delta/2},|\xi|^{-1}\} \lc \bigl(2^j |\xi|\bigr)^{-\sigma},
\]
where $\sigma = \frac{d-1}{d+3} > 0$.
Substituting this estimate into \eqref{e:31kjgeq}, we obtain
\Be\label{e:31kjgeqmore}
\begin{split}
\eqref{e:31kjgeq}&\lc\sum_{k\in\Z}\tau\int\Big(\sum_{j\geq c_{k,n,s,1} }2^{-2(j-k)\sigma}\Big)\Big(\sum_{j \geq 0}  |\tilde{\beta}_{k}(\xi)\widehat{g_{n-j,s}}(\xi)|^2\Big)d\xi
\\
&\lc 2^{-(n+s)2\sigma\epsilon}\sum_{j \geq 0}\|g_{n-j,s}\|_{L_2(\A)}^2
\lc2^{-(n+s)2\sigma\epsilon}\lam\|f\|_{L_1(\A)},
\end{split}
\Ee
where in the second inequality we apply the fact that the supports of $\tilde\beta_{k}$ are essentially disjoint, the Plancherel theorem, and the last inequality follows from Lemma \ref{l:19gbasic}.

Finally combining \eqref{e:28L2}, \eqref{e:28Kj11}, \eqref{e:28jleqk}, \eqref{e:31kjgeq} and \eqref{e:31kjgeqmore}, we conclude that
\Bes
\big\|\sum_j  T^\nu_jg_{n-j,s}\big\|^2_{L_2(\A)}\lc 2^{-(n+s)\alp_1}\lam\|f\|_{L_1(\A)}
\Ees
where the constant $\alp_1=\min\{2(1-\epsilon), 2\sigma\epsilon\}$, which completes the proof of Proposition \ref{l:31good}.

\section{Proof of the estimates related to the bad functions}\label{s:315}

\vskip0.24cm

In the proof of this part, there is no need to make a finite microlocal decomposition of the kernel. We split $Tb$ into four components:
\Bes
(1_\A-\zeta)Tb(1_\A-\zet)+\zet Tb(1_\A-\zet)+(1_\A-\zet)Tb\zet+\zet Tb\zet.
\Ees

By Lemma \ref{l:19excep} and the same reduction argument used for the good functions, it suffices to show that the desired estimate holds for $\vphi \big(|\zet Tb\zet|>\lam/4\big)$.

From the definition of $T$, we further rewrite $Tb$ as follows: for any $x\in\R^d$,
\Bes
\begin{split}
\zet(x)Tb(x)\zet(x)&=\zet(x)\sum_{j \geq 0}T_j\biggl[\sum_{s\geq0}\sum_{n\in\mathbb{Z}}b_{n-j,s}\biggr](x)\zet(x)
  =\sum_{s\geq0}\sum_{n\in\mathbb{Z}}\sum_{j \geq 0}\zet(x)T_jb_{n-j,s}(x)\zet(x)\\
  &=\sum_{s\geq0}\sum_{n\geq100}\sum_{j \geq 0}\zet(x)T_jb_{n-j,s}(x)\zet(x),
\end{split}
\Ees
where the last equality follows because when $\zet(x)\neq0$ and $n<100$, we have
$$\zet(x)T_jb_{n-j,s}(x)\zet(x)=0,$$
by an argument analogous to that in \eqref{e:28exec}.

By  the Chebyshev  inequality, the triangle inequality and the fact  $\zet$ is a projection in $\A$, we then obtain
\Bes
\begin{split}
\vphi \big(|\zet T  b \zet|>\lam/4 \big)&\lc\lam^{-2}\|\zet T b \zet\|_{L_2(\A )}^2\\
&\lc\lam^{-2}\bigg(\sum_{s\geq 0}\sum_{n\geq100}\Big\|\sum_{j \geq 0} T_j b_{n-j,s}\Big\|_{L_2(\A)}\bigg)^2.
\end{split}
\Ees
Consequently, to complete the treatment of the bad functions, it is enough to establish the following result.

\begin{proposition}\label{l:31bad}
Let $n \geq 100$ and $s \geq 0$, there exists a positive constant $\alpha_2$ such that
\[
\Big\|{\sum_{j \geq 0} T_j b_{n-j,s}}\Big\|_{L_2(\A)}^2 \lesssim 2^{-(n+s)\alpha_2} \lambda \Norm{f}_{L^1(\A)}.
\]
\end{proposition}

To prove Proposition \ref{l:31bad}, we first state an auxiliary H\"older continuity estimate of the $TT^*$ kernel.
\begin{lemma}\label{l:31TT8}
Set $\tilde K_j(x) = K_j(-x)$. Let $x \in \mathbb{R}^d\setminus\{0\}$ and $|h|<|x|/2$.  Then we have
\Be\label{e:31KKj}
|{\tilde{K_j}*K_j(x+h)-\tilde{K_j}*K_j(x)}| \lesssim  |2^{-j}h|^{1/3}|2^{-j}x|^{-1/2}2^{-jd},
\Ee
and
\Be\label{e:31KKji}
|{\tilde{K_j}*K_i(x+h)-\tilde{K_j}*K_i(x)}| \lesssim |2^{-i}h|^{1/2}2^{-jd} \quad \text{for}\quad i\leq j-10.
\Ee
\end{lemma}

Here we point out \eqref{e:31KKj} is from M. Christ and Rubio de Francia \cite[Lemma 6.1]{CR88} and \eqref{e:31KKji} is from S. Sato \cite[Lemma, Page 61]{Sat95} in the case $j=0$. The general case $j \geq 0$ can be proved similarly or via the dilation arguments as noted by M. Christ \cite{Chr88} and S. Sato \cite{Sat95}.

\medskip

\subsection{Proof of Proposition \ref{l:31bad}: Setup}
By the definition of $b_{n-j,s}$ and the triangle inequality, we get
\Bes
\begin{split}
\Big\|{\sum_{j \geq 0} T_j b_{n-j,s}}\Big\|_{L_2(\A)} &\leq \Big\|{\sum_{j \geq 0} T_j p_{n-j}(f-f_{n-j+s})p_{n-j+s}}\Big\|_{L_2(\A)}\\
&\quad+\Big\|{\sum_{j \geq 0} T_j p_{n-j+s}(f-f_{n-j+s})p_{n-j}}\Big\|_{L_2(\A)}.
\end{split}
\Ees
In the following we only consider the first term on the right side above.
Write $$\Big\|\sum_{j \geq 0} T_jp_{n-j}(f-f_{n-j+s})p_{n-j+s}\Big\|^2_{L_2(\A)}=:I+II$$
where $I$ equals
\Bes
\begin{split}
\sum_{j \geq 0}\sum_{i\geq0: |i-j|<10}\tau\int\big(T_jp_{n-j} (f-f_{n-j+s}) p_{n-j+s}(x)\big)^*
T_ip_{n-i} (f-f_{n-i+s}) p_{n-i+s}(x)dx
\end{split}
\Ees
and $II$ equals
\Bes
\begin{split}
\sum_{j\geq 0}\sum_{i\geq0:|i-j|\geq10}\tau\int\big(T_jp_{n-j} (f-f_{n-j+s}) p_{n-j+s}(x)\big)^*
T_ip_{n-i} (f-f_{n-i+s}) p_{n-i+s}(x)dx.
\end{split}
\Ees

Therefore it suffices to show that $I$ and $II$ satisfy the desired estimate in Proposition \ref{l:31bad}.

\medskip

\subsection{\textbf{Estimates of the term $I$.}} First notice that by the Cauchy-Schwarz inequality, we get
$$|I|\lc \sum_{j \geq 0} \big\|T_jp_{n-j}(f-f_{n-j+s})p_{n-j+s}\big\|^2_{L_2(\A)}.$$

Applying the Fubini theorem to change the order of integration  and the H\"older inequality, we get

\Bes
\begin{split}
&\quad\big\|T_jp_{n-j}(f-f_{n-j+s})p_{n-j+s}\big\|^2_{L_2(\A)}\\
&=\tau\int{p_{n-j+s} (f-f_{n-j+s}) p_{n-j}}(x)\Big[\tilde{K}_j*K_j*{(p_{n-j} (f-f_{n-j+s}) p_{n-j+s})}(x)\Big]dx\\
&\leq\|p_{n-j+s} (f-f_{n-j+s}) p_{n-j}\|_{L_1(\A)}\Big\|\tilde{K}_j*K_j*{(p_{n-j} (f-f_{n-j+s}) p_{n-j+s})}\Big\|_{L_\infty(\A)}.
\end{split}
\Ees

\medskip

Before proceeding with our proof, we first claim that
\Be\label{e:31pfpexpec}
\sum_{j \geq 0}\|p_{n-j+s} (f-f_{n-j+s}) p_{n-j}\|_{L_1(\A)} \lesssim \Norm{f}_{L^1(\A)},
\Ee
and for any $j \geq 0$, there exists $\alpha_3>0$ such that
\begin{equation}\label{e:31Kjjinfty}
\big\| \tilde{K}_j*K_j*{(p_{n-j} (f-f_{n-j+s}) p_{n-j+s})}\big\|_{L_\infty(\A)} \lesssim 2^{-(n+s)\alpha_3 }\lam.
\end{equation}
Once we have Claims \eqref{e:31pfpexpec} and \eqref{e:31Kjjinfty}, the desired bound for the term $I$ can be easily established.

\medskip

Let us now prove Claim \eqref{e:31pfpexpec}.
Using the triangle inequality, we get
\Bes
\|p_{n-j+s} (f-f_{n-j+s}) p_{n-j}\|_{L_1(\A)} \lesssim \|p_{n-j+s} f p_{n-j}\|_{L_1(\A)}+\|p_{n-j+s} f_{n-j+s} p_{n-j}\|_{L_1(\A)}.
\Ees
By  the H\"older inequality and $\sum_{j\in\Z}p_j=1-q$, we obtain
\Bes
\begin{split}
\sum_{j \geq 0}\|{p_{n-j+s} f p_{n-j}}\|_{L_1(\A)}&\leq\Big(\sum_{j \geq 0}\|p_{n-j+s} f ^{1/2}\|_{L_2(\A)}^2\Big)^{\frac12}
\Big(\sum_{j \geq 0}\| f ^{1/2}p_{n-j}\|_{L_2(\A)}^2\Big)^{\frac12}\\
&=\Big(\sum_{j \geq 0}\vphi(p_{n-j+s}f)\Big)^{\frac12}\Big(\sum_{j \geq 0}\vphi(p_{n-j}f)\Big)^{\frac12}\leq\|f\|_{L^1(\A)}.
\end{split}
\Ees
Similarly by the trace invariance of conditional expectations, we also get
\Bes
\begin{split}
\sum_j\|{p_{n-j+s} f_{n-j+s} p_{n-j}}\|_{L_1(\A)}\leq\|f\|_{L^1(\A)}.
\end{split}
\Ees
Combining the above two estimates, we get exactly  \eqref{e:31pfpexpec}.

\medskip

Next we turn to Claim \eqref{e:31Kjjinfty}. Applying the cancellation condition:  $\int_Q p_{n-j}(f - f_{n-j+s})p_{n-j+s}dx=0$ for any dyadic cube $Q\in\mathcal{Q}_{n-j+s}$,  we get that for $x\in\R^d$
\Be
\begin{split}
&\quad\tilde{K}_j*K_j*{(p_{n-j} (f-f_{n-j+s}) p_{n-j+s})}(x)\\
&=\sum_{Q\in\mathcal{Q}_{n-j+s}} \int_Q \bigl(\tilde{K}_j*K_j(x-y)-\tilde{K}_j*K_j(x-y_Q)\bigr) p_{Q_s}(f(y)-f_Q )p_Q dy,
\end{split}
\Ee
where $y_Q$ is the center of dyadic cube $Q$, ${Q_s}\in\mathcal{Q}_{n-j}$ is the ancestral dyadic cube of $Q\in\mathcal{Q}_{n-j+s}$. To simplify the notation, we set
$$K_{j,j,Q}(x,y)=\tilde{K}_j*K_j(x-y)-\tilde{K}_j*K_j(x-y_Q).$$

To prove Claim \eqref{e:31Kjjinfty}, we need to split the kernel $K_{j,j,Q}(x,y)$ into near-piece  and far-piece via characteristic functions:
\begin{align*}
K_{j,j,Q}(x,y)
&= K_{j,j,Q}(x,y)\chi_{\{\abs{x-y} \leq \sqrt{d}\cdot 2^{j-n-s}\}}
+ K_{j,j,Q}(x,y)\chi_{\{\abs{x-y} > \sqrt{d}\cdot 2^{j-n-s}\}} \\
&=: K_{j,j,Q}^1(x,y) + K_{j,j,Q}^2(x,y).
\end{align*}

\medskip

In the following, we first treat the near-piece kernel $K_{j,j,Q}^1(x,y)$. More precisely, we need to give uniform $L_\infty(\A)$ estimates for the following two terms
\[
A_1(x) = \sum_{Q\in \Q_{n-j+s}} \int_Q K_{j,j,Q}^1(x,y) p_{Q_s} f(y) p_Q dy,
\]
and
\[
A_2(x) = \sum_{Q\in \Q_{n-j+s}} \int_Q K_{j,j,Q}^1(x,y) p_{Q_s} f_Q p_Q dy.
\]

We only treat $A_1(x)$ since $A_2(x)$ can be handled by a similar argument. In fact, there is an average over $f$ in the term $A_2(x)$ so the argument will be slightly easier since our proof for the term $A_1$ below will try to add an average for $f$.
Using the Hilbert-valued H\"older inequality in Lemma \ref{l:31holder-col}, we factor the norm:
\[
\begin{aligned}
\Norm{A_1(x)}_{\M}
\leq{}& \biggl\| \biggl(\sum_{Q\in \Q_{n-j+s}} \int_Q |{K_{j,j,Q}^1(x,y)} | p_{Q_s} f(y) p_{Q_s} dy \biggr)^{\frac12} \biggr\|_{\M} \\
&\cdot \biggl\| \biggl(\sum_{Q\in \Q_{n-j+s}} \int_Q |{K_{j,j,Q}^1(x,y)}| p_Q f(y) p_Q dy \biggr)^{\frac12} \biggr\|_{\M}.
\end{aligned}
\]

By the trivial estimate of $K_j$ and the definition of $K_{j,j,Q}^1$, we get
$$|K_{j,j,Q}^1(x,y)|\lc2^{-jd}\chi_{\{\abs{x-y}\leq \sqrt{d}2^{j-n-s}\}}.$$

Then using this estimate, we majorize the inner sum as follows
\Be\label{e:31A_11}
\begin{split}
&\sum_{Q\in \Q_{n-j+s}} \int_Q |{K_{j,j,Q}^1(x,y)}| p_{Q_s} f(y) p_{Q_s} dy\lc \sum_{\substack{Q\in \Q_{n-j+s} \\ \dist(x,Q)\lesssim 2^{j-n-s}}}  2^{-jd}\int_Q p_{Q_s} f(y) p_{Q_s} dy\\
&\leq \sum_{\substack{Q_s\in \Q_{n-j} \\ \dist(x,Q_s)\lesssim 2^{j-n-s}}}  2^{-jd}\int_{Q_s} p_{Q_s} f(y) p_{Q_s} dy=\sum_{\substack{Q\in \Q_{n-j} \\ \dist(x,Q)\lesssim 2^{j-n-s}}}  2^{-jd} p_{Q} f_{Q} p_{Q} |Q|.
\end{split}
\Ee

Next we apply pointwise estimate of the projection $p_Q$ in the noncommutative Calder\'on-Zygmund decomposition (see (ii) in Lemma \ref{l:28cucu}): $p_{Q} f_{Q} p_{Q}\lc\lam p_Q$, the above expression is bounded by
\Be\label{e:31A_11f}
2^{-jd}\lambda \sum_{\substack{Q\in \Q_{n-j} \\ \dist(x,Q)\lesssim 2^{j-n-s}}} |Q| \lesssim \lambda 2^{-nd},
\Ee
which follows from the fact that all dyadic cubes considered in the sum are disjoint and their union is contained in a ball centered at $x$ with radius less than $2^{j-n-s}+2^{j-n}\lc2^{j-n}$.

Similarly, we can obtain
\Be\label{e:31A_12}
\begin{split}
&\sum_{Q\in \Q_{n-j+s}} \int_Q |{K_{j,j,Q}^1(x,y)}| p_{Q} f(y) p_{Q} dy\\
&\lc
\sum_{\substack{Q\in \Q_{n-j+s} \\ \dist(x,Q)\lesssim 2^{j-n-s}}}  2^{-jd} p_{Q} f_{Q} p_{Q} |Q|\lc\lam2^{-(n+s)d},
\end{split}
\Ee
where the second inequality follows from  $p_{Q} f_{Q} p_{Q}\lc\lam p_Q$, the fact that all dyadic cubes considered are disjoint and their union is contained in a ball centered at $x$ with radius less than $2^{j-n-s}$.

Combining these  estimates \eqref{e:31A_11}, \eqref{e:31A_11f} and \eqref{e:31A_12}, we conclude
\[
\Norm{A_1(x)}_{\M} \lesssim \lambda 2^{-(2n+s)\frac{d}{2}}\lc\lambda 2^{-(n+s)\frac{d}{2}}.
\]
A similar estimate also holds for $A_2(x)$.

\medskip

It remains to handle the far-piece kernel $K_{j,j,Q}^2(x,y)$.
Thus, we need to give uniform $L_\infty(\A)$ estimates of the following two terms:
\[
B_1(x) = \sum_{Q\in \Q_{n-j+s}} \int_Q K_{j,j,Q}^2(x,y) p_{Q_s} f(y) p_Q dy,
\]
and
\[
B_2(x) = \sum_{Q\in \Q_{n-j+s}} \int_Q K_{j,j,Q}^2(x,y) p_{Q_s} f_Q p_Q dy.
\]
We only prove the bound for $B_1(x)$ since $B_2(x)$ can be dealt with similarly.

We first give a pointwise estimate for $K_{j,j,Q}^2(x,y)$. The H\"older continuity of $\tilde{K}_{j}*K_j$ in Lemma \ref{l:31TT8} plays a key role here.  Since $y\in Q$ and $y_Q$ is the center of $Q$, we get
$
|{y-y_Q}| \leq \frac{\sqrt{d}}{2} \cdot 2^{j-n-s}.
$
Then by the support of $K_{j,j,Q}^2(x,y)$, we obtain $|y-y_Q|\leq \fr12|x-y|$.
Using the H\"older continuity of $\tilde{K}_{j}*K_j$ in Lemma \ref{l:31TT8},
 we derive the following pointwise decay estimate,
\Be\label{e:31Kjjdec}
\begin{split}|{K_{j,j,Q}^2(x,y)}|
&\lesssim 2^{-jd} \bigl(2^{-j}\abs{y-y_Q}\bigr)^{\frac13} \bigl(2^{-j}\abs{x-y}\bigr)^{-\frac12}\chi_{\{\sqrt{d}2^{j-n-s}<|x-y|<2^{j+1}\}} \\
&\lesssim 2^{-(n+s)\frac13}2^{-jd} \bigl(2^{-j}\abs{x-y}\bigr)^{-\frac12}
\chi_{\{\sqrt{d}2^{j-n-s} < \abs{x-y} < 2^{j+1}\}}.
\end{split}
\Ee

Below we give an $L_\infty$ estimate for  $B_1(x)$. Using the Hilbert-valued H\"older inequality in Lemma \ref{l:31holder-col}, we obtain
\Be\label{e:31B1}
\begin{split}
\Norm{B_1(x)}_{\M}
\leq & \biggl\| \biggl(\sum_{Q\in \Q_{n-j+s}} \int_Q |{K_{j,j,Q}^2(x,y)}| p_Q f(y) p_Q dy \biggr)^{\frac12} \biggr\|_{\M}\\
&\cdot \biggl\| \biggl(\sum_{Q\in \Q_{n-j+s}} \int_Q |{K_{j,j,Q}^2(x,y)}| p_{Q_s} f(y) p_{Q_s} dy \biggr)^{\frac12} \biggr\|_{\M}.
\end{split}
\Ee

Consider the first term on the right side of \eqref{e:31B1}. Utilizing \eqref{e:31Kjjdec}, we get
\Be\label{e:31B11}
\begin{split}
\sum_{Q\in \Q_{n-j+s}} &\int_Q |{K_{j,j,Q}^2(x,y)}| p_Q f(y) p_Q dy\\
&\lc 2^{-(n+s)\fr13}2^{-jd}\sum_{\substack{Q\in \Q_{n-j+s} \\ \dist(x,Q)\lesssim 2^{j}}}p_Q\cdot\int_{\mathfrak{m}(x,Q)}\fr{f(y)}{(2^{-j}|x-y|)^{\fr12}}dy\cdot p_Q,
\end{split}
\Ee
where $\mathfrak{m}(x,Q)=\{y\in Q: |x-y|\geq \sqrt{d}2^{j-n-s}\}$. Notice that for any $y\in Q$ we have $\abs{y-y_Q}\leq \tfrac12\abs{x-y}$, so $\abs{x-y}\approx \abs{x-y_Q}$. Therefore the above estimate is majorized by
\Be\label{e:31B12}
\begin{split}
&2^{-(n+s)\fr13}2^{-jd}\sum_{\substack{Q\in \Q_{n-j+s} \\ 2^{j-n-s}\lc|x-y_Q|\lesssim 2^{j}}}{(2^{-j}|x-y_Q|)^{-\fr12}}p_Q f_Q p_Q |Q|\\
&\lc 2^{-(n+s)\fr13}2^{-jd}\lambda\sum_{k=0}^{n+s}\sum_{\substack{Q\in \Q_{n-j+s} \\ |x-y_Q|\approx 2^{j-k}}}{(2^{-j}|x-y_Q|)^{-\fr12}} |Q|\\
&\lesssim 2^{-(n+s)\frac13}2^{-jd}\lambda \sum_{k=0}^{n+s} 2^{ \frac12 k}2^{(j-k)d}
\lesssim 2^{-(n+s)\frac13}\lambda,
\end{split}
\Ee
where the first inequality follows from the  property $p_Q f_Q p_Q\lc \lam p_Q$ while  the second inequality relies on disjointness of dyadic cubes and  all these cubes $Q$ with $|x-y_Q|\approx 2^{j-k}$ lie inside a ball of radius $2^{j-k}+2^{j-n-s}\lc 2^{j-k}$.

Next we consider the second term on the right side of \eqref{e:31B1}. Utilizing \eqref{e:31Kjjdec}, we get
\Be\label{e:31B13}
\begin{split}
\sum_{Q\in \Q_{n-j+s}} &\int_Q |{K_{j,j,Q}^2(x,y)}| p_{Q_s} f(y) p_{Q_s} dy\\
&\lc 2^{-(n+s)\fr13}2^{-jd}\sum_{\substack{Q\in \Q_{n-j+s} \\ \dist(x,Q)\lesssim 2^{j}}}p_{Q_s}\cdot\int_{\mathfrak{m}(x,Q)}\fr{f(y)}{(2^{-j}|x-y|)^{\fr12}}dy\cdot p_{Q_s}\\
&\lc 2^{-(n+s)\fr13}2^{-jd}\sum_{Q_s\in \Q_{n-j}}\sup_{\substack{y\in Q_s \\ 2^{j-n-s}\lc|x-y|\lc2^j }}(2^{-j}|x-y|)^{-\fr12} p_{Q_s} f_{Q_s} p_{Q_s}|Q_s|\\
&\lc 2^{-(n+s)\fr13}2^{-jd}\lam\sum_{Q_s\in \Q_{n-j}}\sup_{\substack{y\in Q_s \\ 2^{j-n-s}\lc|x-y|\lc2^j }}\fr{|Q_s|}{(2^{-j}|x-y|)^{\fr12}} ,
\end{split}
\Ee
where the last inequality follows from $p_{Q_s} f_{Q_s} p_{Q_s}\lc\lam p_{Q_s}$.
For any $y\in Q_s$, $|y-y_{Q_s}|\lc2^{j-n}$. If $|x-y_{Q_s}|\approx2^{j-k}$ for some $0\leq k\leq n$, then $|x-y|\approx|x-y_{Q_s}|$.
Therefore we obtain
\Be\label{e:31B14}
\begin{split}
&\sum_{Q_s\in \Q_{n-j}}\sup_{\substack{y\in Q_s \\ 2^{j-n-s}\lc|x-y|\lc2^j }}\fr{|Q_s|}{(2^{-j}|x-y|)^{\fr12}}=\sum_{k=0}^{n+s}\sum_{\substack{ {Q_s\in \Q_{n-j}} \\2^{j-n-s}\lc|x-y| \\|x-y_{Q_s}|\approx2^{j-k} }}\sup_{y\in Q_s}\fr{|Q_s|}{(2^{-j}|x-y|)^{\fr12}}\\
&\lc\sum_{k=0}^{n}\sum_{\substack{ {Q_s\in \Q_{n-j}} \\ |x-y_{Q_s}|\approx2^{j-k} }}\fr{|Q_s|}{(2^{-j}|x-y_{Q_s}|)^{\fr12}}+\sum_{k=n+1}^{n+s}\sum_{\substack{ {Q_s\in \Q_{n-j}} \\ |x-y_{Q_s}|\approx2^{j-k} }}{(2^{n+s})^{\fr12}}|Q_s|\\
&\lc\sum_{k=0}^{n}2^{\fr12 k+(j-k)d}+\sum_{k=n+1}^{n+s}(2^{\fr12(n+s)+(j-n)d})
\lc2^{jd}(1+s2^{\fr12s})
\end{split}
\Ee
where the second-to-last inequality follows from the fact that all dyadic cubes are disjoint and their union lies inside a ball with radii $2^{j-k}$ for $k=0,\cdots,n$ or $2^{j-n}$ for $k=n+1,\cdots, n+s$.

Combining these estimates \eqref{e:31B11}, \eqref{e:31B12}, \eqref{e:31B13} and \eqref{e:31B14}, we conclude
\[
\Norm{B_1(x)}_{\M} \lesssim \lambda 2^{-\frac{1}{24}(n+s)}.
\]
A similar estimate holds for $B_2(x)$.

\medskip

Finally combining these estimates for $A_1(x)$, $A_2(x)$, $B_1(x)$ and $B_2(x)$, we finish the proof of Claim \eqref{e:31Kjjinfty} by choosing $\alp_3=\fr{1}{24}$.
Hence the proof related to these estimates for the term $I$ is complete.

\medskip

\subsection{\textbf{Estimates of the term $II$}}
Rewrite $II$ as the sum of the following two terms
\Bes
\begin{split}
\sum_{j\geq0}\sum_{i\geq0:i\leq j-10}\tau\int\big(T_jp_{n-j} (f-f_{n-j+s}) p_{n-j+s}(x)\big)^*
T_ip_{n-i} (f-f_{n-i+s}) p_{n-i+s}(x)dx
\end{split}
\Ees
and
\Bes
\begin{split}
\sum_{i\geq0}\sum_{j\geq0:j\leq i-10}\tau\int T_ip_{n-i} (f-f_{n-i+s}) p_{n-i+s}(x)\big(T_jp_{n-j} (f-f_{n-j+s}) p_{n-j+s}(x)\big)^*
dx
\end{split}
\Ees
where the second expression follows from the trace property $\tau(ab)=\tau(ba)$. We only give an estimate for the first term above since the second term is symmetric.

By the H\"older inequality in Lemma \ref{l:31holder-col}, we get
\Bes
\begin{split}
&\quad\Big|\sum_{i\geq0:i\leq j-10}\tau\int\big(T_jp_{n-j} (f-f_{n-j+s}) p_{n-j+s}(x)\big)^*
T_ip_{n-i} (f-f_{n-i+s}) p_{n-i+s}(x)dx\Big|\\
&=\Big|\tau\int{p_{n-j+s} (f-f_{n-j+s}) p_{n-j}}(x)\Big[\sum_{i\leq j-10}\tilde{K}_j*K_i*{(p_{n-i} (f-f_{n-i+s}) p_{n-i+s})}(x)\Big]dx\Big|\\
&\leq\|p_{n-j+s} (f-f_{n-j+s}) p_{n-j}\|_{L_1(\A)}\Big\|\sum_{i\leq j-10}\tilde{K}_j*K_i*{(p_{n-i} (f-f_{n-i+s}) p_{n-i+s})}\Big\|_{L_\infty(\A)}.
\end{split}
\Ees

We claim that there exists a positive constant $\alp_4$ such that for any fixed $j\geq0$,
\Be\label{e:31Kjiuni}
\Big\|\sum_{i\leq j-10}\tilde{K}_j*K_i*{(p_{n-i} (f-f_{n-i+s}) p_{n-i+s})}\Big\|_{L_\infty(\A)}\lc\lam2^{-(n+s)\alp_4}.
\Ee
Suppose Claim \eqref{e:31Kjiuni} is true, then combined with \eqref{e:31pfpexpec}, we immediately get the desired bound for the term $II$.

In the following we start to prove Claim \eqref{e:31Kjiuni}. Utilizing the cancellation condition:  $\int_Q p_{n-i}(f - f_{n-i+s})p_{n-i+s}dx=0$ for any dyadic cube $Q\in\mathcal{Q}_{n-i+s}$,   we get that for $x\in\R^d$
\Be
\begin{split}
&\quad\tilde{K}_j*K_i*{(p_{n-i} (f-f_{n-i+s}) p_{n-i+s})}(x)\\
&=\sum_{Q\in\mathcal{Q}_{n-i+s}} \int_Q \bigl(\tilde{K}_j*K_i(x-y)-\tilde{K}_j*K_i(x-y_Q)\bigr) p_{Q_s}(f(y)-f_Q )p_Q dy,
\end{split}
\Ee
where $y_Q$ is the center of dyadic cube $Q$, ${Q_s}\in\mathcal{Q}_{n-i}$ is the ancestral dyadic cube of $Q\in\mathcal{Q}_{n-i+s}$. To simplify the notation, we set
$$K_{j,i,Q}(x,y)=\tilde{K}_j*K_i(x-y)-\tilde{K}_j*K_i(x-y_Q).$$

To prove Claim \eqref{e:31Kjiuni}, we need to split the kernel $K_{j,i,Q}(x,y)$ into near-piece  and far-piece via characteristic functions:
\begin{align*}
K_{j,i,Q}(x,y)
&= K_{j,i,Q}(x,y)\chi_{\{\abs{x-y} \leq \sqrt{d}\cdot 2^{j-n-s}\}}
+ K_{j,i,Q}(x,y)\chi_{\{\abs{x-y} > \sqrt{d}\cdot 2^{j-n-s}\}} \\
&=: K_{j,i,Q}^1(x,y) + K_{j,i,Q}^2(x,y).
\end{align*}

\medskip

In the following, we first treat the near-piece kernel $K_{j,i,Q}^1(x,y)$. We will give uniform $L_\infty(\A)$ estimates for the following two terms
\[
C_1(x) := \sum_{i\leq j-10}\sum_{Q\in \Q_{n-i+s}} \int_Q K_{j,i,Q}^1(x,y) p_{Q_s} f(y) p_Q dy,
\]
and
\[
C_2(x) := \sum_{i\leq j-10}\sum_{Q\in \Q_{n-i+s}} \int_Q K_{j,i,Q}^1(x,y) p_{Q_s} f_Q p_Q dy.
\]

We only treat $C_1(x)$ since $C_2(x)$ can be handled by a similar argument.
Using the Hilbert-valued H\"older inequality in Lemma \ref{l:31holder-col}, we obtain
\[
\begin{aligned}
\Norm{C_1(x)}_{\M}
\leq & \biggl\| \biggl(\sum_{i\leq j-10}\sum_{Q\in \Q_{n-i+s}} \int_Q |{K_{j,i,Q}^1(x,y)} | p_{Q_s} f(y) p_{Q_s} dy \biggr)^{\frac12} \biggr\|_{\M} \\
&\cdot \biggl\| \biggl(\sum_{i\leq j-10}\sum_{Q\in \Q_{n-i+s}} \int_Q |{K_{j,i,Q}^1(x,y)}| p_Q f(y) p_Q dy \biggr)^{\frac12} \biggr\|_{\M}.
\end{aligned}
\]

By the trivial estimate of $K_i$ and the definition of $K_{j,i,Q}^1$, we get
\Be\label{e:31KjiQ}
|K_{j,i,Q}^1(x,y)|\lc2^{-jd}\chi_{\{\abs{x-y}\leq \sqrt{d}2^{j-n-s}\}}.
\Ee

Using this estimate, we majorize the inner sum as follows
\Be\label{e:31C_11}
\begin{split}
&\sum_{i\leq j-10}\sum_{Q\in \Q_{n-i+s}} \int_Q |{K_{j,i,Q}^1(x,y)}| p_{Q_s} f(y) p_{Q_s} dy\\
&\lc \sum_{i\leq j-10}\sum_{\substack{Q\in \Q_{n-i+s} \\ \dist(x,Q)\lesssim 2^{j-n-s}}}  2^{-jd}\int_Q p_{Q_s} f(y) p_{Q_s} dy\\
&\leq \sum_{i\leq j-10}\sum_{\substack{Q_s\in \Q_{n-i} \\ \dist(x,Q_s)\lesssim 2^{j-n-s}}}  2^{-jd}\int_{Q_s} p_{Q_s} f(y) p_{Q_s} dy\\
&=\sum_{i\leq j-10}\sum_{\substack{Q\in \Q_{n-i} \\ \dist(x,Q)\lesssim 2^{j-n-s}}}  2^{-jd} p_{Q} f_{Q} p_{Q} |Q|.
\end{split}
\Ee

Next we apply the noncommutative Calder\'on-Zygmund decomposition property (see (ii) in Lemma \ref{l:28cucu}) $p_{Q} f_{Q} p_{Q}\lc\lam p_Q$, the above expression is bounded by
\Be\label{e:31C_11f}
\begin{split}
2^{-jd}\lambda \sum_{i\leq j-10}\sum_{\substack{Q\in \Q_{n-i} \\ \dist(x,Q)\lesssim 2^{j-n-s}}} p_Q|Q|&=2^{-jd}\lambda \sum_{i\leq j-10}\sum_{\substack{Q\in \Q_{n-i} \\ \dist(x,Q)\lesssim 2^{j-n-s}}} \int_{Q} p_{n-i}(y)dy\\
&=2^{-jd}\lambda \sum_{i\leq j-10} \int_{E_{j,i,1}(x)} p_{n-i}(y)dy,
\end{split}
\Ee
where
$$E_{j,i,1}(x)=\bigcup_{\substack{Q\in \Q_{n-i} \\ \dist(x,Q)\lesssim 2^{j-n-s}}}Q.$$
Notice that all dyadic cubes considered in $E_{j,i,1}(x)$ are disjoint. Then for any $i\leq j-10$, $E_{j,i,1}(x)$ is contained in a ball centered at $x$ with radius less than $2^{j-n-s}+2^{i-n}\lc2^{j-n}$ which  is denoted by $\mathfrak{B}_{j-n}(x)$ since it is independent of $i$.
Hence we get
\Be\label{e:31C_11fe}
\sum_{i\leq j-10} \int_{E_{j,i,1}(x)} p_{n-i}(y)dy\leq  \int_{\mathfrak{B}_{j-n}(x)} \sum_{i\leq j-10} p_{n-i}(y)dy\leq  |\mathfrak{B}_{j-n}(x)|\lc2^{jd-nd},
\Ee
where the second inequality follows from the fact that all $p_k$s are disjoint projections so $\sum_k p_{k}\leq 1$.

Therefore, combining \eqref{e:31C_11}, \eqref{e:31C_11f} and \eqref{e:31C_11fe}, we get
\Be\label{e:31C111}
\sum_{i\leq j-10}\sum_{Q\in \Q_{n-i+s}} \int_Q |{K_{j,i,Q}^1(x,y)} | p_{Q_s} f(y) p_{Q_s} dy\lc\lam2^{-nd}.
\Ee

Similarly, applying \eqref{e:31KjiQ} and $p_{Q} f_{Q} p_{Q}\lc\lam p_Q$,  we can obtain
\Be\label{e:31C_12}
\begin{split}
&\sum_{i\leq j-10}\sum_{Q\in \Q_{n-i+s}} \int_Q |{K_{j,i,Q}^1(x,y)}| p_{Q} f(y) p_{Q} dy\\
&\lc
\sum_{i\leq j-10}\sum_{\substack{Q\in \Q_{n-i+s} \\ \dist(x,Q)\lesssim 2^{j-n-s}}}  2^{-jd} p_{Q} f_{Q} p_{Q} |Q|\\
&=2^{-jd}\lambda \sum_{i\leq j-10}\sum_{\substack{Q\in \Q_{n-i+s} \\ \dist(x,Q)\lesssim 2^{j-n-s}}} \int_{Q} p_{n-i+s}(y)dy\\
&=2^{-jd}\lambda \sum_{i\leq j-10} \int_{E_{j,i,2}(x)} p_{n-i+s}(y)dy,
\end{split}
\Ee
where
$$E_{j,i,2}(x)=\bigcup_{\substack{Q\in \Q_{n-i+s} \\ \dist(x,Q)\lesssim 2^{j-n-s}}}Q.$$
Observe that all dyadic cubes considered in $E_{j,i,2}(x)$ are disjoint.
For any $i\leq j-10$, we see that $E_{j,i,2}(x)$ is contained in a ball centered at $x$ with radius less than $2^{j-n-s}$ which  is denoted by $\mathfrak{B}_{j-n-s}(x)$ since it is independent of $i$.
Hence we obtain
\Be\label{e:31C_12fe}
\begin{split}
\sum_{i\leq j-10} \int_{E_{j,i,2}(x)} p_{n-i+s}(y)dy&\leq  \int_{\mathfrak{B}_{j-n-s}(x)} \sum_{i\leq j-10} p_{n-i+s}(y)dy\\
&\leq  |\mathfrak{B}_{j-n-s}(x)|\lc2^{jd-nd-sd},
\end{split}
\Ee
where the second inequality follows from $\sum_k p_k\leq1$.

Therefore, combining \eqref{e:31C_12} and \eqref{e:31C_12fe}, we have
\Be\label{e:31C112}
\sum_{i\leq j-10}\sum_{Q\in \Q_{n-i+s}} \int_Q |{K_{j,i,Q}^1(x,y)} | p_{Q} f(y) p_{Q} dy\lc\lam2^{-(n+s)d}.
\Ee

Now combining these  estimates \eqref{e:31C111}  and \eqref{e:31C112}, we conclude
\[
\Norm{C_1(x)}_{\M} \lesssim \lambda 2^{-(n+s)\frac{d}{2}}.
\]
A similar estimate also holds for $C_2(x)$.

\medskip

In the following we handle the far-piece kernel $K_{j,i,Q}^2(x,y)$.
Thus, we need to give uniform $L_\infty(\A)$ estimates of the following two terms:
\[
D_1(x) :=\sum_{i\leq j-10} \sum_{Q\in \Q_{n-i+s}} \int_Q K_{j,i,Q}^2(x,y) p_{Q_s} f(y) p_Q dy,
\]
and
\[
D_2(x) :=\sum_{i\leq j-10} \sum_{Q\in \Q_{n-i+s}} \int_Q K_{j,i,Q}^2(x,y) p_{Q_s} f_Q p_Q dy.
\]
We only prove the bound for $D_1(x)$ since $D_2(x)$ can be dealt with similarly.

Let us first give a useful estimate for $K_{j,i,Q}^2(x,y)$. The H\"older continuity of $\tilde{K}_{j}*K_i$ in Lemma \ref{l:31TT8} plays a crucial role.  For $y\in Q$ and $y_Q$ is the center of $Q$, we get
$
|{y-y_Q}| \leq \frac{\sqrt{d}}{2} \cdot 2^{i-n-s}.
$
Then by the support of $K_{j,i,Q}^2(x,y)$, we derive that $|y-y_Q|\leq 2^{-11}|x-y|$ since $i\leq j-10$.
Applying the H\"older continuity of $\tilde{K}_{j}*K_i$ in Lemma \ref{l:31TT8},
 we obtain the pointwise decay estimate,
\Be\label{e:31Kjidec}
\begin{split}|{K_{j,i,Q}^2(x,y)}|
&\lesssim 2^{-jd} \bigl(2^{-i}\abs{y-y_Q}\bigr)^{\frac12} \chi_{\{\sqrt{d}2^{j-n-s}<|x-y|<2^{j+1}\}}.
\end{split}
\Ee

We start to treat the term related to $D_1(x)$. Using the Hilbert-valued H\"older inequality in Lemma \ref{l:31holder-col}, we obtain
\Be\label{e:31D1}
\begin{split}
\Norm{D_1(x)}_{\M}
\leq & \biggl\| \biggl(\sum_{i\leq j-10}\sum_{Q\in \Q_{n-i+s}} \int_Q |{K_{j,i,Q}^2(x,y)}| p_Q f(y) p_Q dy \biggr)^{\frac12} \biggr\|_{\M}\\
&\cdot \biggl\| \biggl(\sum_{i\leq j-10}\sum_{Q\in \Q_{n-i+s}} \int_Q |{K_{j,i,Q}^2(x,y)}| p_{Q_s} f(y) p_{Q_s} dy \biggr)^{\frac12} \biggr\|_{\M}.
\end{split}
\Ee

Consider the first term on the right side of \eqref{e:31D1}. Utilizing \eqref{e:31Kjidec}, we get
\Be\label{e:31D11}
\begin{split}
\sum_{i\leq j-10}\sum_{Q\in \Q_{n-i+s}} &\int_Q |{K_{j,i,Q}^2(x,y)}| p_Q f(y) p_Q dy\\
&\lc 2^{-(n+s)\fr12}2^{-jd}\sum_{i\leq j-10}\sum_{\substack{Q\in \Q_{n-i+s} \\ \dist(x,Q)\lesssim 2^{j}}}p_Q f_Q p_Q |Q|\\
&=2^{-(n+s)\fr12}2^{-jd}\lambda \sum_{i\leq j-10} \int_{E_{j,i,3}(x)} p_{n-i+s}(y)dy,
\end{split}
\Ee
where
$$E_{j,i,3}(x)=\bigcup_{\substack{Q\in \Q_{n-i+s} \\ \dist(x,Q)\lesssim 2^{j}}}Q.$$
We point out that all dyadic cubes considered in $E_{j,i,3}(x)$ are disjoint. Then for any $i\leq j-10$, $E_{j,i,3}(x)$ is contained in a ball centered at $x$ with radius less than $2^{j}$ which  is denoted by $\mathfrak{B}_{j}(x)$ since it is independent of $i$.
Hence we have
\Be\label{e:31D_11fe}
\begin{split}
\sum_{i\leq j-10} \int_{E_{j,i,3}(x)} p_{n-i+s}(y)dy&\leq  \int_{\mathfrak{B}_{j}(x)} \sum_{i\leq j-10} p_{n-i+s}(y)dy\leq  |\mathfrak{B}_{j}(x)|\lc2^{jd},
\end{split}
\Ee
where the second inequality follows from $\sum_{k}p_k\leq1$.

Therefore by \eqref{e:31D11} and \eqref{e:31D_11fe}, we conclude that
\Be\label{e:31D11f}
\sum_{i\leq j-10}\sum_{Q\in \Q_{n-i+s}} \int_Q |{K_{j,i,Q}^2(x,y)}| p_Q f(y) p_Q dy\lc\lam2^{-\fr12(n+s)}.
\Ee

It remains to consider the second term on the right side of \eqref{e:31D1}. Applying \eqref{e:31Kjidec}, we obtain
\Be\label{e:31D12}
\begin{split}
\sum_{i\leq j-10}\sum_{Q\in \Q_{n-i+s}} &\int_Q |{K_{j,i,Q}^2(x,y)}| p_{Q_s} f(y) p_{Q_s} dy\\
&\lc 2^{-(n+s)\fr12}2^{-jd}\sum_{i\leq j-10}\sum_{\substack{Q\in \Q_{n-i+s} \\ \dist(x,Q)\lesssim 2^{j}}}p_{Q_s}\int_Q f(y)dy p_{Q_s}\\
&\lc 2^{-(n+s)\fr12}2^{-jd}\sum_{i\leq j-10}\sum_{\substack{Q_s\in \Q_{n-i} \\ \dist(x,Q_s)\lesssim 2^{j}}}p_{Q_s} f_{Q_s} p_{Q_s} |Q_s|\\
&=2^{-(n+s)\fr12}2^{-jd}\lambda \sum_{i\leq j-10} \int_{E_{j,i,4}(x)} p_{n-i}(y)dy,
\end{split}
\Ee
where
$$E_{j,i,4}(x)=\bigcup_{\substack{Q\in \Q_{n-i} \\ \dist(x,Q)\lesssim 2^{j}}}Q.$$
Since all dyadic cubes considered in $E_{j,i,4}(x)$ are disjoint, we get that for any $i\leq j-10$, $E_{j,i,4}(x)$ is contained in a ball centered at $x$ with radius less than $2^{j}$ which  is denoted by $\mathfrak{B}_{j}(x)$.
Hence we have
\Be\label{e:31D_12fe}
\begin{split}
\sum_{i\leq j-10} \int_{E_{j,i,4}(x)} p_{n-i}(y)dy&\leq  \int_{\mathfrak{B}_{j}(x)} \sum_{i\leq j-10} p_{n-i}(y)dy\leq  |\mathfrak{B}_{j}(x)|\lc2^{jd},
\end{split}
\Ee
where in the second  inequality we use $\sum_k{p_k}\leq1$.

By \eqref{e:31D12} and \eqref{e:31D_12fe}, we conclude that
\Be\label{e:31D12f}
\sum_{i\leq j-10}\sum_{Q\in \Q_{n-i+s}} \int_Q |{K_{j,i,Q}^2(x,y)}| p_{Q_s} f(y) p_{Q_s} dy\lc\lam2^{-\fr12(n+s)}.
\Ee

Combining these estimates \eqref{e:31D11f} and \eqref{e:31D12f}, we obtain
\[
\Norm{D_1(x)}_{\M} \lesssim \lambda 2^{-\frac{1}{2}(n+s)}.
\]
A similar estimate holds for $D_2(x)$.

\medskip

Finally combining these estimates for $C_1(x)$, $C_2(x)$, $D_1(x)$ and $D_2(x)$, we finish the proof of Claim \eqref{e:31Kjiuni} by choosing $\alp_4=\fr{1}{2}$.
Hence we get the desired estimates for the term $II$.

\medskip

\section*{Declarations}

\textbf{Competing interests} The author declares no competing interests.

\medskip

\textbf{Funding} This work is supported by National Natural Science Foundation of China (No. 12271124, No. 12322107 and No. W2441002) and Heilongjiang Provincial Natural Science Foundation of China (YQ2022A005).

\medskip

\textbf{Data Availability} No datasets were generated or analysed during the current study.


\vskip1cm

\bibliographystyle{amsplain}
\bibliography{rf}

\end{document}